\documentclass[11pt]{article}
\usepackage[margin=2.15cm]{geometry} 
\usepackage{graphicx} 
\usepackage[colorlinks=false]{hyperref} 
\usepackage{amssymb} 
\usepackage{amsfonts} 
\usepackage{amsmath,amscd} 
\usepackage{amsthm}  
\usepackage{mathtools} 
\usepackage{titlesec} 
\usepackage{csquotes} 
\usepackage[shortlabels]{enumitem} 
\usepackage{subfiles} 
\usepackage{tikz-cd} 
\usepackage{wrapfig} 
\usepackage{authblk} 
\usepackage{cutwin}
\usepackage{tabularx} 
\usepackage{hyperref}
\hypersetup{
     colorlinks   = true,
     citecolor    = green
}

\usepackage{nomencl}

\makenomenclature

\usepackage{afterpage}

\setlist[itemize]{itemsep=0mm} 
\numberwithin{equation}{section} 
\titleformat*{\section}{\Large \scshape\center} 
\titleformat*{\subsection}{\fontsize{14}{14} \sffamily} 

\theoremstyle{plain}
\newtheorem{theorem}{Theorem}[section]
\newtheorem*{theorem*}{Theorem} 
\newtheorem{lemma}[theorem]{Lemma}
\newtheorem{proposition}[theorem]{Proposition}
\newtheorem{corollary}[theorem]{Corollary}

\theoremstyle{definition}
\newtheorem{definition}[theorem]{Definition}
\newtheorem{example}[theorem]{Example}

\newcommand{\Addresses}{{
  \bigskip
  \footnotesize

  \textsc{Department of Mathematical Sciences, Norwegian University of Science and Technology,\\ 7491 Trondheim, Norway.}\par\nopagebreak
  \textit{E-mail addresses}: \texttt{franz.luef@ntnu.no} and  \texttt{eirik.berge@ntnu.no}
}}

\theoremstyle{remark}
\newtheorem*{remark}{Remark}

\begin{document}
\pagenumbering{gobble}
\title{\Huge{A Large Scale Approach to Decomposition Spaces}}
\author{Eirik Berge and Franz Luef}
\date{}
\maketitle

\pagenumbering{arabic}

\begin{abstract}
Decomposition spaces are a class of function spaces constructed out of ``well-behaved" coverings and partitions of unity of a set. The structure of the covering of the set determines the properties of the decomposition space. Besov spaces, shearlet spaces and modulation spaces are well-known decomposition spaces. In this paper we focus on the geometric aspects of decomposition spaces and utilize that these are naturally captured by the large scale properties of a metric space, the covered space, associated to a covering of a set. We demonstrate that decomposition spaces constructed out of quasi-isometric covered spaces have many geometric features in common. The notion of geometric embedding is introduced to formalize the way one decomposition space can be embedded into another decomposition space while respecting the geometric features of the coverings. Some consequences of the large scale approach to decomposition spaces are (i) comparison of coverings of different sets, (ii) study of embeddings of decomposition spaces based on the geometric features and the symmetries of the coverings and (iii) the use of notions from large scale geometry, such as asymptotic dimension or hyperbolicity, to study the properties of decomposition spaces. 

We draw some consequences of our general investigations for the modulation spaces and for a new class of decomposition spaces based on the special linear group of the Euclidean plane. These results are based on considerations of the large scale properties of stratified Lie groups, locally compact groups and Euclidean spaces, where we utilize the growth type of a group as a large scale invariant. 

\end{abstract}

\section{Introduction}

Large scale geometry has its origins in the seminal work of Gromov in \cite{Gromov_Hyperbolic, Gromov} and has led to substantial progress in group theory, operator algebras and geometry. In this paper we add another item to the long list of applications of large scale geometry: the theory of function spaces, in particular the decomposition spaces of Feichtinger and Gr\"obner \cite{Hans_Grobner,Hans_2}. The link between decomposition spaces and coarse geometry has also been pointed out in the Ph.D. thesis of Koch \cite{ThesisRene}.

Several function spaces in time-frequency analysis and harmonic analysis possess a description through a geometric decomposition of the domain space. These spaces are referred to as \textit{decomposition spaces} and contain among them Besov spaces and modulation spaces. Since the inception of decomposition spaces in \cite{Hans_Grobner}, a fundamental question has been to decide whether one decomposition space embeds into another decomposition space. These investigations have mostly been considered when the two decomposition spaces in question consist of functions/distributions on the same underlying space; an exception is the tour de force paper \cite{Felix_main} where many results treat the case where the underlying spaces are different open subsets of the same ambient Euclidean space with non-empty intersection. We will investigate embeddings of a geometric nature between decomposition spaces defined on different sets by utilizing methods from large scale geometry. \par
Let us briefly sketch the construction of decomposition spaces, see Section \ref{Chapter_Decomposition_Spaces}. Let $\mathcal{Q} = (Q_i)_{i \in I}$ be a well-behaved covering on a set $X$ and consider a partition of unity $\Phi = (\varphi_i)_{i \in I}$ subordinate to the covering $\mathcal{Q}$. Decomposition spaces consist of functions $f:X \to \mathbb{C}$ that have nice local behaviour with respect to the partition $\Phi$ measured in terms of a Banach space $(B,\|\cdot\|_B)$: This local information is encoded in the sequence \[f_i := \|f \cdot \varphi_{i}\|_{B}, \qquad i \in I.\] 
Furthermore we want to ensure global regularity of $f$, which we obtain by imposing the sequence $(f_i)_{i \in I}$ to be an element of a suitably chosen sequence space $(Y,\|\cdot\|_Y)$. Hence, the {\it decomposition space} $\mathcal{D}(\mathcal{Q},B,Y)$  is the space of functions such that the norm 
\begin{equation}
\label{decomposition norm}
    \|f\|_{\mathcal{D}(\mathcal{Q},B,Y)} := \|(f_i)_{i \in I}\|_{Y},
\end{equation}
is finite. \par
The way to relate decomposition spaces with large scale geometry is to associate to any well-behaved covering $\mathcal{Q}$ on $X$ a metric space $(X,d_{\mathcal{Q}})$. The metric $d_{\mathcal{Q}}(x,y)$ essentially counts the minimum number of borders of the sets $Q_i$ one need to cross when going from $x$ to $y$. The important features of the covering $\mathcal{Q}$ are encapsulated in the metric space structure of $(X,d_{\mathcal{Q}})$. 

Recall that a map $f:(X,d_X) \to (Z,d_Z)$ between metric spaces is called a \textit{quasi-isometric embedding} if there exist constants $L,C > 0$ such that \[ \frac{1}{L}d_{X}(x,y) - C \leq d_{Z}(f(x),f(y)) \leq Ld_{X}(x,y) + C,\] for all $x,y \in X$. Quasi-isometric embeddings are generalizations of isometric embeddings that allow the spaces to be locally different as long as they have the same global behavior. This leads us to consider maps $\phi:(X,\mathcal{Q}) \to (Z,\mathcal{P})$ between sets equipped with well-behaved coverings that are quasi-isometric embeddings with respect to the distances $d_{\mathcal{Q}}$ and $d_{\mathcal{P}}$. 

There is a standard notion of \textit{equivalence} between coverings $\mathcal{Q},\mathcal{P}$ on the same space $X$ present in the literature on decomposition spaces \cite{Labate2013, Hans_Grobner, Felix_main, Hartmut_Besov}. We give a new proof of Proposition \ref{generalization_of_equivalent_coverings} stating that the coverings $\mathcal{Q}$ and $\mathcal{P}$ are equivalent if and only if the identity map $Id_X:(X,d_{\mathcal{Q}}) \to (X,d_{\mathcal{P}})$ is a bijective quasi-isometric embedding, that is, a \textit{quasi-isometry}. The statement goes back to the paper \cite{Hans_Grobner} and has been recently proved in a special case in the Ph.D. thesis of Koch \cite{ThesisRene} where its purpose was to compare decomposition spaces with coarse geometric methods. This framework provides a natural extension of equivalent coverings to coverings defined on different sets. \par
The functorial way of associating the metric space $(X,d_{\mathcal{Q}})$ to the space $X$ equipped with the well-behaved covering $\mathcal{Q}$ allows us to consider quasi-isometric invariant properties of the covering $\mathcal{Q}$. In particular, we discuss the \textit{asymptotic dimension}, \textit{growth type}, and \textit{quasi-hyperbolicity} of a well-behaved covering. These properties will be used time and time again in later sections to simplify arguments already present in the literature. \par There is a canonical way of associating to a path-connected, locally compact group $G$ a covering $\mathcal{U}(G)$ reflecting the group operation introduced in \cite{Hans_2}. We will show in Theorem \ref{uniform_structure} that we can reduce the problem of understanding the covering $\mathcal{U}(G)$ to the study of the asymptotic dimension, the growth type, or the hyperbolicity of certain finitely generated subgroups of $G$. This is explored in more detail for stratified Lie groups in Proposition \ref{Carnot_group_result} and solvable groups in Proposition \ref{solvable_group_result}, where the finitely generated subgroups are respectively nilpotent and strongly polycyclic. For a stratified Lie group $G$ and a lattice $N$ in $G$, we establish in Theorem \ref{growth_vector_result} a correspondence between the growth type of the metric space $(N,d_{\mathcal{U}(G)})$ and the \textit{homogeneous dimension} of the stratified Lie group $G$. \par 
Consider two decomposition spaces $\mathcal{D}(\mathcal{Q},B_1,Y_1)$ and $\mathcal{D}(\mathcal{P},B_2,Y_2)$ related to the coverings $\mathcal{Q}$ and $\mathcal{P}$ on the locally compact spaces $X$ and $Z$, respectively. We will investigate the existence of Banach space embeddings $F:\mathcal{D}(\mathcal{Q},B_1,Y_1) \to \mathcal{D}(\mathcal{P},B_2,Y_2)$ that induce a quasi-isometric embedding between the metric spaces $(X,d_{\mathcal{Q}})$ and $(P,d_{\mathcal{P}})$. These embeddings are called \textit{geometric embeddings} and are introduced in Section \ref{sec: Geometric Embeddings}.
The rest of the paper is devoted to give criteria for when geometric embeddings exist and to examine examples. This relies crucially on the material developed in Section \ref{sec: From Admissible Coverings to the Large Scale Setting} and Section \ref{Chapter_Uniform_Metric_Spaces}. \par 
Two highlights are Proposition \ref{covering_invariance}, showing that geometric embeddings induce quasi-isometric embeddings of the underlying coverings, and Theorem \ref{quasi_implies_embeddings}, showing when quasi-isometries between the metric spaces $(X,d_{\mathcal{Q}})$ and $(Z,d_{\mathcal{P}})$ can induce geometric embeddings between the decomposition spaces $\mathcal{D}(\mathcal{Q},B_1,Y_1)$ and $\mathcal{D}(\mathcal{P},B_2,Y_2)$. \par
In the final section we look at geometric embeddings between well-known decomposition spaces such as the modulation spaces $M^{p,q}(\mathbb{R}^n)$. In Theorem \ref{modulation_space_result} we show that there is a tower of compatible geometric embeddings \[M^{p,q}(\mathbb{R}) \xrightarrow{\Gamma_{1}^{2}}M^{p,q}(\mathbb{R}^2) \xrightarrow{\Gamma_{2}^3} \dots \xrightarrow{\Gamma_{n-1}^{n}} M^{p,q}(\mathbb{R}^n) \xrightarrow{\Gamma_{n}^{n+1}} \dots, \] where there are no geometric embeddings in the other direction. Combining this result with \cite[Theorem 12.2.2]{TF_analysis} shows that there exists a geometric embedding from the Feichtinger algebra $\mathcal{S}_{0}(\mathbb{R}) := M^{1,1}(\mathbb{R})$ to any of the modulation spaces $M^{p,q}(\mathbb{R}^n)$. 
\par
Finally, we consider in Subsection \ref{sec:A Decomposition Space of Hyperbolic Type} the decomposition space \[\mathcal{D}^{p,q}\left(SL(2,\mathbb{R})\right) := \mathcal{D}(\mathcal{U}(SL(2,\mathbb{R})),L^{p},l^q)\] on the semisimple Lie group $SL(2,\mathbb{R})$. The associated metric space $(SL(2,\mathbb{R}),d_{\mathcal{U}})$ is quasi-hyperbolic by Proposition \ref{SL} and we show in Proposition \ref{final_result} that the decomposition space $\mathcal{D}^{p,q}\left(SL(2,\mathbb{R})\right)$ is radically different from the modulation spaces and Besov spaces. \par
In order to make this paper accessible for a broad audience we have included basic results and definitions from large scale geometry. These are given when they are needed rather than including them in an appendix since there are several excellent introductory texts available \cite{Large_Scale,Geometric_Group_Theory}. 
\subsection*{Acknowledgements}
The authors would like to express their gratitude to Hans Feichtinger and Felix Voigtlaender for insightful comments and suggestions. The first author was partially supported by the BFS/TFS project Pure Mathematics in Norway.

\section{From Admissible Coverings to the Large Scale Setting}
\label{sec: From Admissible Coverings to the Large Scale Setting}

\subsection{Covered Spaces and Associated Metric Spaces}

The first order of business is to associate a metric space to any sufficiently nice covering. Let $X$ be a non-empty set. A collection of non-empty subsets $\mathcal{Q} = (Q_i)_{i \in I}$ of $X$ is called an \textit{admissible covering} if it is a covering of $X$ such that \[N_{\mathcal{Q}} := \sup_{i \in I}|i^{*}| < \infty ,\qquad i^* := \left\{j \in I \, \Big| \, Q_{i} \cap Q_{j} \ne \emptyset \right\}.\] We will call the constant $N_{\mathcal{Q}}$ the \textit{admissibility constant} of the covering, while $i^*$ is called the \textit{neighbours} of the index $i$. Analogously, the sets $Q_j$ for $j \in i^*$ are called the \textit{neighbours} of the set $Q_i$. This notion can be inductively extended by $i^{k*} := (i^{(k-1)*})^*$ for $k \geq 2$. Moreover, the abbreviations
\begin{equation}
\label{neighbours}
    Q_{i}^{*} := \bigcup_{j \in i^{*}}Q_j, \qquad Q^{k*} := \left(Q^{(k-1)*}\right)^{*}, \quad k \geq 2
\end{equation} 
will be used to ease the notation. Note that $i \in j^{k*}$ if and only if $j \in i^{k*}$ for all $k \geq 1$. Other elementary properties of neighbours can be found in \cite[Lemma 2.1]{Hans_Grobner} \footnote{The reader should be aware that first statement in \cite[Lemma 2.1]{Hans_Grobner} is false.}. \par 
We call a sequence $Q_{i_1}, \dots, Q_{i_k} \in \mathcal{Q}$ with $x \in Q_{i_1}$ and $y \in Q_{i_k}$ a $\mathcal{Q}$-\textit{chain} from $x$ to $y$ of \textit{length} $k$ whenever $Q_{i_l} \cap Q_{i_{l+1}} \ne \emptyset$ for every $1 \leq l \leq k-1$. The notation $\mathcal{Q}(k,x,y)$ will be used to denote all $\mathcal{Q}$-chains of length $k$ from $x$ to $y$. We will need one additional assumption on admissible coverings so that we can associate to them metric spaces in a natural manner.

\begin{definition}
An admissible covering $\mathcal{Q}$ on a set $X$ will be called a \textit{concatenation} if for every pair of points $x,y \in X$ there exists a positive number $k \in \mathbb{N}$ such that $\mathcal{Q}(k,x,y) \ne \emptyset$. We will refer to the pair $(X,\mathcal{Q})$ as a \textit{covered space} whenever $\mathcal{Q}$ is a concatenation on $X$.
\end{definition}

The notion of a concatenation first appeared in \cite{Hans_Grobner} and is equivalent to the requirement that \[X = \bigcup_{k = 1}^{\infty}Q_{i}^{k*},\] for some (and hence all) $Q_i \in \mathcal{Q}$. 

\begin{definition}
Define the metric $d_{\mathcal{Q}}$ on the covered space $(X,\mathcal{Q})$ by the rule \[d_{\mathcal{Q}}(x,x) = 0, \quad d_{\mathcal{Q}}(x,y) = \inf \left\{k:\mathcal{Q}(k,x,y) \ne \emptyset \right\}, \qquad x,y \in X, \, \, x \ne y.\] The defining properties of a covered space ensure that $(X,d_{\mathcal{Q}})$ is a metric space. We will refer to $(X,d_{\mathcal{Q}})$ as the \textit{associated metric space} to the covered space $(X,\mathcal{Q})$.
\end{definition}

The metric space $(X,d_\mathcal{Q})$ was introduced in \cite{Hans_Grobner} together with a few basic properties. Notice that $(X,d_\mathcal{Q})$ is a uniformly discrete metric space since $d_{\mathcal{Q}}(x,y) \geq 1$ whenever $x$ and $y$ are distinct points. \par 
A common way of comparing two coverings on the same space is as follows: Let $X$ be a set equipped with two admissible coverings $\mathcal{Q} = (Q_i)_{i \in I}$ and $\mathcal{P} = (P_j)_{j \in J}$. We say that $\mathcal{Q}$ is \textit{almost subordinate} to $\mathcal{P}$ and write $\mathcal{Q} \leq \mathcal{P}$ if there exists a $k \in \mathbb{N}$ such that for every $i \in I$ there is a $j \in J$ with $Q_{i} \subset P_{j}^{k*}$. The coverings $\mathcal{Q}$ and $\mathcal{P}$ are said to be \textit{equivalent} if both $\mathcal{Q} \leq \mathcal{P}$ and $\mathcal{P} \leq \mathcal{Q}$ hold. It follows that any admissible covering $\mathcal{Q}$ on a set $X$ is equivalent to the covering $\mathcal{Q}^{k*} := \{Q_i^{k*} \, | \, i \in I\}$ for any $k \geq 1$. So far in the study of decomposition spaces, only coverings on the same set have been compared in the literature. \par

\begin{definition}
A \textit{(metric) net} in a metric space $(X,d_X)$ is a subset $N$ of $X$ such that there exists a constant $M > 0$ with \[\inf_{y \in N}d_{X}\left(x,y\right) \leq M,\] for every $x \in X$. 
A map $f:(X,d_X) \to (Z,d_Z)$ between metric spaces is called a \textit{quasi-isometric embedding} if there exist constants $L,C > 0$ such that \[ \frac{1}{L}d_{X}(x,y) - C \leq d_{Z}(f(x),f(y)) \leq Ld_{X}(x,y) + C,\] for all $x,y \in X$. The constants $L,C$ are called the \textit{parameters} of the quasi-isometric embedding. The map $f$ will be called a \textit{quasi-isometry} if it in addition satisfies that $f(X)$ is a net in $Z$. 
\end{definition}

The notation $(X,d_X) \simeq (Z,d_Z)$ indicates that there exists a quasi-isometry between the metric spaces $(X,d_X)$ and $(Z,d_Z)$. It is common to refer to the quasi-isometry class of a metric space as its \textit{large scale geometry}. A quasi-isometric embedding can have discontinuities; they are however controlled in a uniform manner. We can always choose the parameters $L$ and $C$ of a quasi-isometric embedding to be integers by enlarging them. Two maps $f,g:(X,d_X) \to (Z,d_Z)$ between metric spaces are said to be \textit{close} if there exists a constant $C>0$ such that \[d_{Z}(f(x),g(x)) < C,\] for every $x \in X$. It follows from \cite[Proposition 5.1.10]{Geometric_Group_Theory} that a quasi-isometric embedding $f:(X,d_X) \to (Z,d_Z)$ is a quasi-isometry if and only if there exists a quasi-isometric embedding $g:(Z,d_Z) \to (X,d_X)$ such that $g \circ f$ and $f \circ g$ are close to their respective identity maps. \par

\begin{example}
\label{fin_gen_groups}
It is illustrative to see that the class of metric spaces that can be obtained as the associated metric space of a covered space is rather large. Let $G$ be a finitely generated group with a symmetric generating set $\Sigma$ that contains the identity element of $G$. We obtain a left-invariant metric $d_G$ on $G$ by defining 
\begin{equation}
\label{word_metric}
    d_{G}(g,h) = \min \left\{n \, \Big| \, g^{-1}h = \sigma_{1} \cdots \sigma_{n}, \, \sigma_{i} \in \Sigma\right\}.
\end{equation} 
Consider the covering $\mathcal{Q} = (g\Sigma)_{g \in G}$ on $G$. The admissibility condition is satisfied due the cardinality of the generating set $\Sigma$. To see that $\mathcal{Q}$ is a concatenation it suffices to connect the identity to an arbitrary element $g = \sigma_{1} \cdots \sigma_{k}$ where $\sigma_{i} \in \Sigma$ for $i = 1, \dots, k$. The chain \[\Sigma,\sigma_{1}\Sigma,\sigma_{1}\sigma_{2}\Sigma, \dots, g\Sigma\] connects the identity to $g$ and we have \[\sigma_1 \cdots \sigma_{s+1} \in \left(\sigma_{1} \cdots \sigma_{s}\Sigma \right) \cap \left(\sigma_{1} \cdots \sigma_{s+1}\Sigma\right),\] for every $1 \leq s \leq k-1$. Hence $(G,\mathcal{Q})$ is a covered space where the identity $Id_{G}:(G,d_{\mathcal{Q}}) \to (G,d_{G})$ is a quasi-isometry. Moreover, it follows from \cite[Theorem 1.3.12]{Large_Scale} that any other choice of finite generating set than $\Sigma$ in (\ref{word_metric}) would give a quasi-isometric metric space. 
\end{example}

\begin{proposition}
Let $(X,\mathcal{Q})$ and $(Z,\mathcal{P})$ be covered spaces. Then a map $f:(X,d_{Q}) \to (Z,d_{\mathcal{P}})$ is a quasi-isometric embedding if and only if there exist constants $L,C \in \mathbb{N}$ such that
\begin{equation}
\label{covered_morphism}
    \mathcal{Q}(L(k + C),x,y) \ne \emptyset, \qquad \mathcal{Q}\left(\left\lfloor\frac{k - C}{L}\right\rfloor,x,y\right) = \emptyset,
\end{equation} for every $x,y \in X$, where $k$ is the smallest natural number such that $\mathcal{P}(k,f(x),f(y)) \ne \emptyset$. 
\end{proposition}

\begin{proof}
Let $f:(X,\mathcal{Q}) \to (Z,\mathcal{P})$ be a map that satisfies (\ref{covered_morphism}). Fix $x,y \in X$ and choose the smallest $k \in \mathbb{N}$ such that $d_{\mathcal{P}}(f(x),f(y)) \leq k$. Then $\mathcal{P}(k,f(x),f(y)) \ne \emptyset$ and it follows that \[\mathcal{Q}(L(k + C),x,y) \ne \emptyset.\] Hence $d_{\mathcal{Q}}(x,y) \leq L(k + C)$. The upper bound in the definition of a quasi-isometric embedding is verified similarly. Conversely, let $f:(X,d_\mathcal{Q}) \to (Z,d_{\mathcal{P}})$ be a quasi-isometric embedding with integer parameters $L,C > 0$. Fix $x,y \in X$ and let $k:= d_{\mathcal{P}}(f(x),f(y))$. Then we have \[L(k + C) \geq d_{\mathcal{Q}}(x,y) \geq \frac{k - C}{L}.\] These inequalities imply that $f$ satisfies (\ref{covered_morphism}) by the definition of the distance function $d_{\mathcal{Q}}$.
\end{proof}

A metric space $(X,d)$ is \textit{coarsely connected} if there exists a constant $c > 0$ such that for any two points $x,y \in X$ there exists a sequence $x = x_0,x_1,\dots,x_n = y$ such that $d(x_i,x_{i+1}) \leq c$ for $i = 0, \dots, n-1$. Coarse connectedness is a property that is invariant under quasi-isometries. It is clear from the construction that the associated metric space $(X,d_{\mathcal{Q}})$ of a covered space $(X,\mathcal{Q})$ is coarsely connected.
\begin{example}
Consider $\mathbb{N}_0$ with the metric \[d(n,m) = \max\{n,m\}, \quad \textrm{when} \, n \ne m\] and $d(n,n) = 0$ for any $n,m \in \mathbb{N}_{0}$. Clearly $(\mathbb{N}_{0},d)$ is a uniformly discrete metric space. However, for $m > 1$ we have $d(1,m) = m$ and $d(n,m) \geq m$ for every $n \in \mathbb{N}_{0}$. Since we can pick $m$ arbitrary large the metric space $(\mathbb{N}_{0},d)$ is not coarsely connected. Therefore, the metric space $(\mathbb{N}_{0},d)$ is not quasi-isometric to any associated metric space of a covered space.
\end{example}

\begin{remark}
Let $(X,\mathcal{Q})$ be a covered space with associated metric space $(X,d_{\mathcal{Q}})$. It is often more convenient to work with a smaller metric space; we do this by considering a net $N$ in $(X,d_Q)$. The inclusion $N \hookrightarrow X$ is then a quasi-isometry when we restrict the metric $d_Q$ to the set $N$. We will usually consider nets in $X$ with \textit{bounded geometry}, that is, nets $N$ such that \[|B_{N}(x,r)| \leq \psi(r), \quad r > 0,\] for some function $\psi$ that does not depend on the point $x \in N$. One option for a bounded geometry net $N$ in $(X,d_{\mathcal{Q}})$ is picking a uniformly finite number $F$ of points in each $Q_i \in Q$. Then we have \[|B_N(x,r)| \leq F N_{\mathcal{Q}}^{r}, \quad x \in N, \, r > 0.\] 
\end{remark}

The following proposition originates in the paper \cite[Proposition 3.8 C)]{Hans_Grobner} where it was formulated in terms of bi-Lipschitz equivalences. The fact that any bijective quasi-isometry on a uniformly discrete, bounded geometry metric space is a bi-Lipschitz equivalence \cite[Proposition 9.4.2]{Geometric_Group_Theory} gives the transition between their statement and the one below. Prior to our investigations, a special case of the result \cite[Proposition 3.8 C)]{Hans_Grobner} was proved in the Ph.D. thesis of Ren\'{e} Koch \cite[Theorem 5.2.6]{ThesisRene} containing more details than the original source. We will give a new proof since a detailed proof of the general version of the statement is lacking in the literature.

\begin{proposition}
\label{generalization_of_equivalent_coverings}
Let $(X,\mathcal{Q})$ and $(X,\mathcal{P})$ be covered spaces. Then $\mathcal{Q} \leq \mathcal{P}$ if and only if the identity map $Id_{X}:(X,d_\mathcal{Q}) \to (X,d_\mathcal{P})$ is Lipschitz continuous. Hence the coverings $\mathcal{Q}$ and $\mathcal{P}$ are equivalent if and only if the identity map $Id_X:(X,d_{\mathcal{Q}}) \to (X,d_{\mathcal{P}})$ is a quasi-isometry.
\end{proposition}

\begin{proof}
We start by assuming that $Q$ is almost subordinate to $\mathcal{P}$. For two distinct points $x,y \in X$, there exists a number $M \in \mathbb{N}$ such that $\mathcal{Q}(M,x,y) \neq \emptyset$. Pick a $Q$-chain $Q_{i_1},Q_{i_2}, \dots, Q_{i_M}$ from $x$ to $y$ of length $M$. Then there exists a $k \in \mathbb{N}$ such that for each $l = 1, \dots, M$ we can find an $P_{j(l)} \in \mathcal{P}$ such that \[Q_{i_l} \subset P_{j(l)}^{k*}.\] Since $P_{j(1)}^{k*}$ has non-empty intersection with $P_{j(2)}^{k*}$ we know that \[\textrm{diam}_{\mathcal{P}}\left(P_{j(1)}^{k*} \cup P_{j(2)}^{k*}\right) \leq 2k.\] Continuing this, we obtain by iteration that \[\textrm{diam}_{\mathcal{P}}\left(\bigcup_{l = 1}^{M}P_{j(l)}^{k*}\right) \leq Mk.\] Hence we can find a $\mathcal{P}$-chain between $x$ and $y$ with length at most $Mk$. This shows that \[d_{\mathcal{P}}(x,y) \leq kd_{Q}(x,y),\] and hence the identity map $Id_X:(X,d_{\mathcal{Q}}) \to (X,d_{\mathcal{P}})$ is Lipschitz continuous. \par
Conversely, assume that $d_{\mathcal{P}}(x,y) \leq M d_{Q}(x,y)$ for every $x,y \in X$ and some $M > 0$. We can assume that $M$ is an integer by enlarging it. Fix $x_0 \in X$ and choose $Q_i \in Q$ and $P_j \in \mathcal{P}$ such that $x_0 \in Q_i \cap P_j$. Then for any $y \in Q_i$ we have $d_{Q}(x_0,y) \leq 1$ and thus $d_{\mathcal{P}}(x,y) \leq M.$ Hence there is a $\mathcal{P}$-chain \[P_{j} = P_{j_1}, P_{j_2}, \dots, P_{j_M}\] from $x_0$ to $y$. This shows that $y \in P_{j}^{M*}$ for any $y \in Q_i$ and so $Q_i \subset P_{j}^{M*}$. Since the constant $M$ does not depend on $x$ and $y$ we have that $Q$ is almost subordinate to $\mathcal{P}$.
\end{proof}

The notion of quasi-isometries between associated metric spaces of covered spaces is more flexible than the notion of equivalent coverings since we can compare coverings on different sets. This will allow us to consider quasi-isometric invariant properties of covered spaces through the associated metric space in Subsection \ref{sec: Large Scale Invariants of a Covered Space}. The motivation for considering this is to show that certain decomposition spaces can not embed nicely into other decomposition spaces in Section \ref{Chapter_Decomposition_Spaces} and Section \ref{sec: Concrete Settings}. 

\begin{example}
\label{grid_example}
Consider the \textit{uniform covering} \[\mathcal{U} = (Q_{n_1,\dots,n_k})_{n_1,\dots,n_k \in \mathbb{Z}}, \qquad Q_{n_1,\dots,n_k} := [0,1]^{k} + (n_1,\dots,n_k),\] on $\mathbb{R}^k$.
It is straightforward to check that $\mathcal{U}$ is a concatenation. We will call the resulting metric space $(\mathbb{R}^k,d_\mathcal{U})$ the \textit{uniform metric space} on $\mathbb{R}^k$. The set $\mathbb{Z}^k$ is a net in $(\mathbb{R}^k,d_{\mathcal{U}})$ and we have \[d_{\mathcal{U}}((n_1,\dots,n_k),(m_1,\dots,m_k)) = \max\{|m_1 - n_1|, \dots, |m_k - n_k|\}, \quad (n_1, \dots, n_k), (m_1, \dots, m_k) \in \mathbb{Z}^{k}.\]
In this example a special feature emerges; the integer lattice $\mathbb{Z}^k$ is also a group that acts on itself by isometries when equipped with the metric $d_\mathcal{U}$. Hence the symmetries of the uniform covering $\mathcal{U}$ on $\mathbb{R}^k$ is incorporated in the metric $d_{\mathcal{U}}$ through being left (and right) invariant under the action of $\mathbb{Z}^k$. 
\end{example}

\begin{example}
\label{dyadic covering}
Consider the \textit{dyadic covering} $\mathcal{B} := \mathcal{B}(\mathbb{R}^n) = (D_m)_{m \in \mathbb{N}_{0}}$ on $\mathbb{R}^n$ given by the {\it dyadic intervals} \[D_0 = \left\{x \in \mathbb{R}^n \, \Big | \, \|x\|_{2} \leq 2 \right\}, \qquad D_{m} = \left\{x \in \mathbb{R}^n \, \Big | \, 2^{m-1} \leq \|x\|_{2} \leq 2^{m+1}\right\}, \quad  m \in \mathbb{N},\] where $\|\cdot\|_2$ denotes the Euclidean norm.  As only the magnitude of elements in $\mathbb{R}^n$ determines which dyadic interval they are in, the covering is inherently one-dimensional. Hence by picking the net \[N := \left\{(2^n,\dots,0)  \, \Big| \, n \in \mathbb{N}_{0}\right\},\] we have that $(\mathbb{R}^n,d_{\mathcal{B}(\mathbb{R}^n)}) \simeq (N,d_{\mathcal{B}(\mathbb{R}^n)})$ is quasi-isometric to $\mathbb{N}_0$ with its usual metric. In particular, the metric spaces $(\mathbb{R}^n, d_{\mathcal{B}(\mathbb{R}^n)})$ and $(\mathbb{R}^m,d_{\mathcal{B}(\mathbb{R}^m)})$ are quasi-isometric for all $n,m \geq 1$. 
\end{example}

\subsection{Incorporating the Symmetries of a Covering}

We take a closer look into the symmetries of a covering implemented by group actions, as seen in Example \ref{grid_example}. First we have to introduce some terminology to describe the setting. Let $G$ be a finitely generated group acting on a metric space $(X,d_X)$ by isometries. For $x \in X$ and $R > 0$, the $R$-\textit{stabilizer} $\textrm{Stab}_{R}(x)$ is the set \[\textrm{Stab}_{R}(x) := \left\{g \in G \, \Big | \, d_{X}(gx,x) \leq R \right\}.\] We will call the action of $G$ on $(X,d_{X})$ \textit{large scale stable} if any non-identity element $g \in G$ satisfies \[0 < \sup_{x \in X}d_{X}(gx,x) <\infty.\] Note that a large scale stable action is actually \textit{effective} due to the lower bound, that is, $gx = x$ for every $x \in X$ implies that $g$ is the identity element of $G$. We call a point $x_0 \in X$ \textit{almost transitive} if for every $x \in X$ there exists a $g \in G$ such that \[d_{X}(gx_0,x) \leq C,\] where $C>0$ does not depend on the point $x \in X$. This is a large scale analogue of a transitive action where one allows for some uniform error. Finally, recall that a finitely generated group $N$ is \textit{nilpotent} if its \textit{lower central series} terminates; there should exist $n \in \mathbb{N}_{0}$ such that \[N = C_{0}(N) \rhd C_{1}(N) \rhd \cdots \rhd C_{n}(N) = \{e\}, \quad C_{i}(N) := [N,C_{i-1}(N)], \quad i = 1, \dots, n.\]

\begin{theorem}
\label{symmetries_theorem}
Let $(X,\mathcal{Q})$ be a covered space with associated metric space $(X,d_{\mathcal{Q}})$. Assume there is a large scale stable action of a finitely generated group $G$ on $(X,d_\mathcal{Q})$.

\begin{enumerate}[(a)]
    \item The function \[d_{G}(g,h) := \sup_{x \in X}d_{\mathcal{Q}}(gx,hx), \quad g,h \in G,\] defines a left-invariant metric on $G$. 
    \item Assume that there exists an almost transitive point $x_0 \in X$ such that
\begin{equation}
\label{fixed_point_estimate}
\sup_{x \in X}d_{\mathcal{Q}}(gx,x) \leq Ld_{\mathcal{Q}}(gx_0,x_0) + C
\end{equation}
holds for arbitrary $g \in G$ and uniform constants $L,C > 0$. Then $(G,d_G)$ is quasi-isometric to $(X,d_\mathcal{Q})$.
\item Assume that we have the bound 
\begin{equation}
\label{stabilizer_bound}
\left|\textrm{Stab}_{n}(x)\right| \leq p(n)
\end{equation}
for every $x \in X$ and $n \in \mathbb{N}$, where $p$ is a polynomial with integer coefficients. Then $G$ is quasi-isometric to a finitely generated nilpotent group.
\end{enumerate}  
\end{theorem}

\begin{proof}
\begin{enumerate}[(a)]
\item
The function $d_{G}$ is well-defined by the upper bound in the definition of a large scale stable action. If $d_{G}(g,h) = 0$, then we have $d_{\mathcal{Q}}(gx,hx) = 0$ for every $x \in X$ and the positivity of $d_{\mathcal{Q}}$ implies that $h^{-1}gx = x$ for every $x \in X$. Since the action is effective we conclude that $g = h$. The left-invariance of the metric $d_G$ is a reformulation of the fact the $G$ acts by isometries on $X$. 
\item
Assume there exists an almost transitive point $x_0 \in X$ such that (\ref{fixed_point_estimate}) is satisfied and consider the map $\phi:G \to X$ defined by $\phi(g) = gx_0$. We want to show that $\phi$ is a quasi-isometry between $(G,d_G)$ and $(X,d_{\mathcal{Q}})$. It is tautological that \[d_{\mathcal{Q}}(\phi(g),\phi(h)) \leq d_{G}(g,h).\] Moreover, the estimate (\ref{fixed_point_estimate}) is a simplification of the lower-bound estimate for a quasi-isometric embedding with parameters $L,C > 0$ where the isometry property is incorporated. Finally, the image of $\phi$ is a net because $x_0$ is transitive point.
\item
The $n$-stabilizer bound (\ref{stabilizer_bound}) implies in particular that the metric $d_G$ is \textit{proper}, that is, \[|B_{G}(e,n)|< \infty, \qquad \textrm{for every }n \in \mathbb{N}.\] It follows from \cite[Theorem 1.3.12]{Large_Scale} that all proper, left-invariant metrics on $G$ give quasi-isometric metric spaces. Moreover, Gromov's celebrated Polynomial Growth Theorem \cite{Polynomial_Growth_Paper} implies that the bound (\ref{stabilizer_bound}) is equivalent with $G$ being \textit{virtually nilpotent}, that is, possessing a nilpotent subgroup $N \subset G$ with finite index. The result follows from \cite[Corollary 5.4.5]{Geometric_Group_Theory} stating that finite index subgroups of finitely generated groups are nets.
\end{enumerate} 
\end{proof}

\begin{example}
\label{heisenberg_3_example}
Let $\mathcal{P} = (P_{n,m,l})_{n,m,l \in \mathbb{Z}}$ be the concatenation on $\mathbb{R}^3$ given by \[P_{n,m,l} = (n,m,l) * [0,1]^{3}, \qquad (n,m,l) * (n',m',l') := (n+n',m+m',l + l' + nm').\] This is almost the same as the uniform covering $\mathcal{U}$ on $\mathbb{R}^3$ introduced in Example \ref{grid_example}, except for the intertwining in the third component. It is straightforward to check that the \textit{discrete Heisenberg group} $\mathbb{H}_{3}(\mathbb{Z}) := (\mathbb{Z}^3,*)$ acts on the metric space $(\mathbb{R}^3,d_{\mathcal{P}})$ by isometries. It satisfies all the assumptions in Theorem \ref{symmetries_theorem} (b) and we deduce that the associated metric space $(\mathbb{R}^3,d_{\mathcal{P}})$ is quasi-isometric to the discrete Heisenberg group with any proper, left-invariant metric. We will see after Example \ref{Heisenberg_example} that the concatenation $\mathcal{P}$ on $\mathbb{R}^3$ is not equivalent to the uniform covering $\mathcal{U}$ on $\mathbb{R}^3$. 
\end{example}

\subsection{Large Scale Invariants of a Covered Space}
\label{sec: Large Scale Invariants of a Covered Space}

Let $\mathsf{P}$ denote a quasi-isometric invariant property of a metric space. We say that the covered space $(X,\mathcal{Q})$ has \textit{property} $\mathsf{P}$ if the associated metric space $(X,d_{\mathcal{Q}})$ has property $\mathsf{P}$. The first property we will introduce for covered spaces is a variant of topological dimension adapted to the quasi-isometric setting.

\subsubsection{Asymptotic Dimension}

\begin{definition}
Let $\mathcal{U}$ be a covering of a metric space $(X,d_X)$. The \textit{$R$-multiplicity} of $\mathcal{U}$ for $R>0$ is the smallest integer $n$ such that each ball $B(x,R)$ intersects at most $n$ elements of $\mathcal{U}$ for all $x \in X$. The \textit{asymptotic dimension} of $X$ is the smallest number $n \in \mathbb{N}_{0}$ such that for each $R>0$ there exists a covering $\mathcal{U} = (U_{i})_{i \in I}$ with uniformly bounded diameters and with $R$-multiplicity $n + 1$. If no $n \in \mathbb{N}_{0}$ satisfies the condition, then the metric space $(X,d_X)$ is said to have \textit{infinite asymptotic dimension}. We use the notation $\textrm{asdim}(X,d_{X})$ or simply $\textrm{asdim}(X)$ if the metric is clear from the context.
\end{definition}

The asymptotic dimension is invariant under quasi-isometries, see \cite[Theorem 2.2.5]{Large_Scale}. In particular, if $\mathcal{Q}$ and $\mathcal{P}$ are two concatenations on a set $X$ such that $\textrm{asdim}(X,d_{\mathcal{Q}}) \neq \textrm{asdim}(X,d_{\mathcal{P}})$, then Proposition \ref{generalization_of_equivalent_coverings} implies that $\mathcal{Q}$ and $\mathcal{P}$ are not equivalent coverings. 

\begin{example}
\label{asymptotic_dim_zero}
As an illustration we will show that a covered space has asymptotic dimension zero if and only if it is quasi-isometric to a point. Let $(X,\mathcal{Q})$ be a covered space with asymptotic dimension zero. Consider a net $N \subset X$ formed by picking one element $x_i \in Q_{i}$ for each $i \in I$. It suffices to consider $(N,\mathcal{Q})$ since asymptotic dimension is invariant under quasi-isometries. For $R = 2$ there exists a covering $\mathcal{U} = (U_{j})_{j \in J}$ with uniformly bounded diameters such that $B(x_i,2)$ only intersects one of the $U_{i}$'s for $x_i \in N$. Since $\mathcal{U}$ is a covering it follows that $B(x_i,2) \subset U_j$ for some $j \in J$. If $x_k \in N$ with $d_{\mathcal{Q}}(x_k,x_i) = 1$, then $B(x_k,2)$ also has to be contained in the same $U_j$. Continuing this way shows that $N \subset U_j$ since $\mathcal{Q}$ is a concatenation. Since $U_j$ is bounded it follows that $(N,d_{\mathcal{Q}})$, and hence $(X,d_{\mathcal{Q}})$, is quasi-isometric to a point. Conversely, any bounded metric space clearly has asymptotic dimension zero.
\end{example}

We emphasize that the argument in Example \ref{asymptotic_dim_zero} relies on that $(X,d_{\mathcal{Q}})$ is coarsely connected. The set if $p$-adic numbers $\mathbb{Q}_{p}$ for a prime $p$ has asymptotic dimension zero as a consequence of the inequality \[d_{\mathbb{Q}_p}(x,z) \leq \max \left\{d_{\mathbb{Q}_p}(x,y),d_{\mathbb{Q}_p}(y,z)\right\}, \quad x,y,z \in \mathbb{Q}_{p},\] without being bounded as a metric space.

\begin{proposition}
\label{grids_are_different}
The uniform metric spaces $(\mathbb{R}^n,d_{\mathcal{U}})$ and $(\mathbb{R}^m,d_{\mathcal{U}})$ considered in Example \ref{grid_example} are quasi-isometric only when $n = m$. Moreover, there exists a quasi-isometric embedding from $(\mathbb{R}^n,d_{\mathcal{U}})$ to $(\mathbb{R}^m,d_{\mathcal{U}})$ precisely when $n \leq m$.
\end{proposition}
\begin{proof}
We have already established in Example \ref{grid_example} that $(\mathbb{R}^n,d_{\mathcal{U}})$ is quasi-isometric to the integer lattice $\mathbb{Z}^n$ with its usual left-invariant metric. A standard fact in large scale geometry \cite[Example 2.2.6]{Large_Scale} states that the asymptotic dimension of $\mathbb{Z}^n$ is $n$. Hence the first statement follows from the quasi-isometric invariance of asymptotic dimension. \par 
For the second statement, assume that there is a quasi-isometric embedding $\phi:\mathbb{Z}^n \to \mathbb{Z}^m$. The subspace $\phi(\mathbb{Z}^n) \subset \mathbb{Z}^m$ has to have asymptotic dimension less than $m$ by restricting any covering fulfilling the definition of asymptotic dimension. Hence $\mathbb{Z}^n \simeq \phi(\mathbb{Z}^n)$ implies the necessity of $n \leq m$. If $n \leq m$, the inclusion  $\mathbb{Z}^n \hookrightarrow \mathbb{Z}^m$ into the first $n$ coordinates is easily seen to be a quasi-isometric embedding.
\end{proof}

\begin{example}
\label{example_dyadic_and_grid}
The associated metric space of the dyadic covered space $(\mathbb{R}^n,\mathcal{B}(\mathbb{R}^n))$ considered in Example \ref{dyadic covering} is quasi-isomorphic to $\mathbb{N}_0$ with its usual metric. Since $\mathbb{N}_0 \subset \mathbb{Z}$ and $\mathbb{N}_0$ is not bounded, we can conclude from Example \ref{asymptotic_dim_zero} that the asymptotic dimension of $(\mathbb{R}^n,\mathcal{B}(\mathbb{R}^n))$ is one. Hence the dyadic covering $\mathcal{B}\left(\mathbb{R}^n\right)$ and the uniform covering $\mathcal{U}\left(\mathbb{R}^n\right)$ considered in Example \ref{grid_example} are not equivalent as coverings unless possibly when $n = 1$. However, it follows from a straightforward calculation that there are no quasi-isometries between $\mathbb{N}_0$ and $\mathbb{Z}$ with their usual metrics. Hence the associated metric spaces $(\mathbb{R}^n,d_{\mathcal{U}})$ and $(\mathbb{R}^l,d_{\mathcal{B}(\mathbb{R}^l)})$ are not quasi-isometric for any values $n,l \geq 1$. Although this is rather straightforward to show directly as well, it showcases the potential of the large scale approach.
\end{example}

We showed in Example \ref{fin_gen_groups} that every finitely generated group may be considered as the associated metric space of a covered space. There are examples of finitely generated groups that do not have finite asymptotic dimension, such as the \textit{wreath product} $\mathbb{Z} \wr \mathbb{Z}$. We refer the reader to \cite[Proposition 2.6.3]{Large_Scale} for the definition of wreath product and the calculation giving that $\mathbb{Z} \wr \mathbb{Z}$ has infinite asymptotic dimension.

\subsubsection{Representations as Graphs}

We will associate a graph to any covered space and demonstrate how this makes certain properties of covered spaces more apparent. 
Consider a covered space $(X,\mathcal{Q})$ and form a net $N = (x_i)_{i \in I} \subset X$ where $x_i \in Q_i$ for each $i \in I$. We can consider the graph $G(N)$ whose vertices are indexed by the points in $N$. We declare that there is an edge between the vertices $x_i$ and $x_j$ if and only if $d_{\mathcal{Q}}(x_i,x_j) \leq 2$. Then the metric space $(N,d_\mathcal{Q})$ is quasi-isometric to the usual graph metric on the vertices of $G(N)$, see \cite[Example 1.1.10]{Large_Scale}. Moreover, we can extend the graph metric to the edges by identifying each edge $e = x_i x_j$ with the interval $[0,1]$. The resulting metric space $(G(N),d_G)$ is quasi-isometric to $(X,d_{\mathcal{Q}})$. 
\begin{definition}
A metric space $(X,d_{X})$ is said to be \textit{(quasi-)geodesic} if there exist constants $L,C > 0$ such that for every two points $x,y \in X$ we can find a (quasi-)isometric embedding $\gamma:[0,d_{X}(x,y)] \to X$ with parameters $L,C$ where $\gamma(0) = x$ and $\gamma(d_{X}(x,y)) = y$. 
\end{definition}
Since $(G(N),d_G)$ is a geodesic metric space it follows that $(X,d_{\mathcal{Q}})$ is a quasi-geodesic metric space. The relationship between covered spaces and graph theory is more than superficial, and there is parallel terminology in the two subjects. Recall that the \textit{degree} of a vertex in a graph is the number of neighbouring vertices. A connected graph is said to have \textit{bounded geometry} if the degrees of the vertices are uniformly bounded. Hence the associated metric space of any covered space is quasi-isometric to a connected graph with bounded geometry. This allows us to borrow results from the well established theory of graphs, a connection that to our knowledge has not been made before. In particular, we have the following result from \cite[Example 3.8]{Random_Walks}.

\begin{proposition}
\label{graph_theory_lemma}
Let $(X,\mathcal{Q})$ be any covered space with admissibility constant $N_{\mathcal{Q}} \geq 3$. Then $(X,d_{\mathcal{Q}})$ is quasi-isometric to a connected graph $(G,d)$ equipped with the graph metric and with degrees bounded above by $3$.
\end{proposition}

If the number of elements in each $Q_i$ is larger than $N_{\mathcal{Q}}$, then it is clear from the construction in \cite[Example 3.8]{Random_Walks} that we can take the vertices of $G$ to be elements in $X$ in Proposition \ref{graph_theory_lemma}. The number $3$ is clearly sharp, as any concatenation $\mathcal{Q}$ with $N_{\mathcal{Q}} = 2$ can only have two elements.

\begin{remark}
There is a more general notion than quasi-isometries present in the large scale literature known as \textit{coarse equivalences}, see \cite[Definition 1.4.1]{Large_Scale}. The reason we consider quasi-isometries rather than coarse equivalences follows from the fact that the two definitions coincide between quasi-geodesic metric spaces by \cite[Theorem 1.4.13]{Large_Scale}.
\end{remark}

\subsubsection{Hyperbolicity}

There is a notion of hyperbolicity of a quasi-geodesic metric space that we will use as an invariant of a covered space similarly to asymptotic dimension. First of all, a $(L,C)$ \textit{quasi-geodesic triangle} in a metric space $(X,d_{X})$ is a triple $(\gamma_1, \gamma_2, \gamma_3)$ of quasi-isometric embeddings $\gamma_i: [0,L_i] \to X$ with parameters $L,C > 0$ such that \[\gamma_1(L_1) = \gamma_{2}(0), \quad \gamma_2(L_2) = \gamma_{3}(0), \quad \gamma_{3}(L_3) = \gamma_{1}(0).\] We call such a quasi-geodesic triangle $\delta$\textit{-slim} if there exists $\delta > 0$ such that \[\textrm{Im}(\gamma_i) \subset \bigcup_{x \in \textrm{Im}(\gamma_j) \cup \textrm{Im}(\gamma_k)}B(x,\delta),\] where $i,j,k \in \{1,2,3\}$ are all distinct.

\begin{definition}
Let $(X,d_{X})$ be a quasi-geodesic metric space. We say that $(X,d_{X})$ is \textit{quasi-hyperbolic} if there exist constants $L,C, \delta > 0$ such that every $(L',C')$ quasi-geodesic triangle in $(X,d_{X})$ is $\delta$-slim for all $L' \geq L$ and $C' \geq C$. 
\end{definition}

Note that quasi-hyperbolicity is a quasi-isometric invariant by \cite[Proposition 7.2.9]{Geometric_Group_Theory}. Hence we can declare a covered space $(X,\mathcal{Q})$ to be quasi-hyperbolic if the associated metric space $(X,d_{\mathcal{Q}})$ is quasi-hyperbolic. If a finitely generated group $G$ is quasi-hyperbolic with any (hence all) proper, left-invariant metric, it is common in the literature to simply call it a \textit{hyperbolic group} and we will follow this convention. We will now present basic results regarding quasi-hyperbolic metric space assembled from \cite[Chapter 7]{Geometric_Group_Theory} that will be used in Subsection \ref{sec: More Examples} and Subsection \ref{sec:A Decomposition Space of Hyperbolic Type}.

\begin{lemma}
\label{hyperbolicity_lemma}
\begin{enumerate}[(a)]
    \item Let $(X,d_X)$ be a quasi-geodesic metric space and $(Z,d_Z)$ a quasi-hyperbolic metric space. Then the existence of a quasi-isometric embedding $\phi:(X,d_X) \to (Z,d_Z)$ implies that $(X,d_X)$ is also quasi-hyperbolic.
    \item The hyperbolic plane $\mathbb{H}^2 = \{(x,y) \in \mathbb{R}^2 \, | \, y > 0\}$ with its usual hyperbolic metric is quasi-hyperbolic.
    \item Among the groups $\mathbb{Z}^n$, only the group $\mathbb{Z}$ is hyperbolic. Moreover, if $G$ is any hyperbolic group and $g \in G$ has infinite order, then the map \[\psi:(\mathbb{Z},d_{\mathbb{Z}}) \longrightarrow (G,d_G), \qquad n \longmapsto g^n \] is a quasi-isometric embedding.
    \item Any group that contain a subgroup isomorphic to $\mathbb{Z}^2$ is not hyperbolic. 
\end{enumerate}
\end{lemma}

\section{Uniform Metric Spaces on Locally Compact Groups}
\label{Chapter_Uniform_Metric_Spaces}

In this section we investigate coverings on path-connected, locally compact groups that reflect the group structure. While starting generally, we quickly focus in on stratified Lie groups and solvable Lie groups to obtain concrete examples. Finally, we examine a hyperbolic covering on the special linear group $SL(2,\mathbb{R})$. In Section \ref{Chapter_Decomposition_Spaces} we will start to build decomposition spaces on top of these coverings. The metric space machinery developed in Section \ref{sec: From Admissible Coverings to the Large Scale Setting} together with results in this section will be used in Subsection \ref{sec: Geometric Embeddings} and Section \ref{sec: Concrete Settings} to show that certain embeddings between different decomposition spaces are impossible. 

\subsection{Uniform Metric Spaces}

We begin by recalling some basic definitions related to locally compact groups. A \textit{locally compact group} $G$ is a locally compact Hausdorff space with a group structure such that the multiplication and inversion are continuous maps. A subset $A \subset G$ is called \textit{symmetric} if $A^{-1} = A$, where $A^{-1} := \{y^{-1} \,| \, y \in A \}$. One can always find a symmetric and precompact neighbourhood of the identity on a locally compact group $G$ by considering $U^{-1}U$, where $U$ is a precompact neighbourhood of the identity. \par 
On any locally compact group $G$ there exists a unique \textit{left Haar measure} $\mu$ up to scaling, that is, a non-zero Radon measure satisfying $\mu(gE) = \mu(E)$ for $E \subset G$ and $g \in G$. The analogous statement also holds true for right Haar measures. Locally compact groups where the right and left Haar measure coincide are called \textit{unimodular}. We will later consider the spaces $L^{p}(G) := L^{p}(G,\mathcal{B},\mu)$ for $p \in [1,\infty)$, where $\mathcal{B}$ is the Borel sigma-algebra and $\mu$ is a fixed left Haar measure. \par 
In Subsections \ref{sec: Carnot Groups} and \ref{sec: More Examples} we will consider \textit{lattices} in locally compact groups $G$; they are discrete subgroups $\Gamma$ in $G$ such that there exists a $G$-invariant Borel measure $\mu_{G/\Gamma}$ on the quotient $G/\Gamma$ with $\mu_{G/\Gamma}(G/\Gamma) < \infty$. The prototypical example to have in mind is the lattice $\mathbb{Z}^n$ inside the locally compact group $\mathbb{R}^n$. The concrete examples considered in Section \ref{sec: Concrete Settings} will all be Lie groups. We refer the readers to \cite{Folland} and \cite{Warner} for basic material about locally compact groups and Lie groups, respectively. \par

Let $G$ be a locally compact group that is path-connected and fix a Haar measure $\mu$ on $G$. We will associate to $G$ a metric space that will reflect the group structure. Pick a precompact and symmetric set $Q_{0} \subset G$ with non-void interior called a \textit{reference set} and consider the \textit{continuous covering} $\{gQ_{0}\}_{g \in G}$ in the language of \cite{Hans_2}. The precompactness of $Q_{0}$ insures that $\mu(Q_{0}) < \infty$ while the non-void interior guarantees that $0 < \mu(Q_{0})$. It follows from the symmetry of $Q_{0}$ and \cite[Theorem 4.1 (A)]{Hans_2} that there exist elements $\{g_i\}_{i \in I}$ in $G$ such that $\mathcal{U} := \mathcal{U}(G) := \{g_i Q_{0} \}_{i \in I}$ defines an admissible covering on $G$. \par 
We simplify the notation $Q_i := g_{i} Q_{0}$ and assume without loss of generality that $g_0 = e$ to make the notation compatible with the one already in place for the reference set $Q_{0}$. Furthermore, we have from  \cite[Theorem 4.1 (B)]{Hans_2} that any other family $\{h_j\}_{j \in J}$ in $G$ with the same property defines an equivalent covering. Moreover, the specific choice of the reference set $Q_0$ is easily seen to be irrelevant. Hence we can always choose $Q_0$, and hence $Q_i$, to be open if we so desire. We refer to $\mathcal{U}(G)$ as the \textit{uniform covering} of the path-connected, locally compact group $G$. Notice that this notation is compatible with Example \ref{grid_example} since $\mathcal{U}(\mathbb{R}^n)$ is the uniform covering on $\mathbb{R}^n$.

\begin{lemma}
The uniform covering of any path-connected, locally compact group $G$ is a concatenation.
\end{lemma}
\begin{proof}
Fix $g,h \in G$ and let $\gamma:[0,1] \to G$ denote a continuous path such that $\gamma(0) = g$ and $\gamma(1) = h$. We choose $Q_0$ to be open and consider the sets \[U_{i} := \gamma^{-1}(Q_{i} \cap \textrm{Im}(\gamma)), \quad i \in I.\] The collection $(U_i)_{i \in I}$ form an open covering of $[0,1]$ and the compactness of the interval $[0,1]$ implies that there exists a finite sub-covering $U_{i_{1}}, \dots, U_{i_n}$. Thus \[\textrm{Im}(\gamma) \subset  \bigcup_{l = 1}^{n} Q_{i_{l}},\] and we have that $\mathcal{U}(x,y,n) \ne \emptyset.$ The necessity of requiring that $G$ is path-connected follows from considering $G = \mathbb{Z}_2$.
\end{proof}
In our language, we obtain that $(G,\mathcal{U})$ is a covered space such that the quasi-isometry class of $(G,d_{\mathcal{U}})$ does not depend on the construction. We will call the resulting metric space $(G,d_\mathcal{U})$ the \textit{uniform metric space} on the path-connected, locally compact group $G$. We make the convention that a covering $\mathcal{U}$ on a path-connected, locally compact group $G$ is assumed to be the uniform covering unless stated otherwise. 

\begin{remark}
Uniform metric spaces have also been considered by Ren\'e Koch in his Ph.D. Thesis \cite{ThesisRene} through a slightly different construction: The author defines a metric $d_W$ on any locally compact group $G$ by fixing a symmetric and precompact unit neighbourhood $W$ and defining the distance $d_{W}(x,y)$ between two distinct points $x,y \in G$ to be the minimal number $m$ such that $yx^{-1} \in W^{m}$. This description is convenient and makes it obvious that the resulting metric $d_{W}$ on $G$ is left-invariant. The reader should be aware that \cite{ThesisRene} allows the metric to take infinite values as he also consider locally compact groups that are not necessarily path-connected. 
\end{remark}

A metric $d$ on a set $X$ is said to be \textit{proper} if the balls induced by $d$ are precompact. This coincides with our use of the term proper in the proof of Theorem \ref{symmetries_theorem} and in Example \ref{heisenberg_3_example}. The following result shows that we can sometimes understand the uniform metric space on path-connected, locally compact groups by understanding the large scale geometry of a finitely generated subgroup.

\begin{theorem}
\label{uniform_structure}
Let $G$ be a path-connected, locally compact group and let $d$ be a proper, left-invariant metric on $G$ that is compatible with the topology on $G$. Assume $N$ is a finitely generated subgroup of $G$ that is a net in $G$ and that $d$ restricts to a locally finite metric on $N$. Then the uniform metric space $(G,d_{\mathcal{U}})$ is quasi-isometric to the space $(N,d)$ where $d$ is any proper, left-invariant metric on $N$.
\end{theorem}

\begin{proof}
Since $N$ is a net in $G$ we can find a constant $M > 0$ such that \[\mathcal{U} = \{nB(e,M)\}_{n \in N} = \{B(n,M)\}_{n \in N}\] is a covering on $G$. By picking a left Haar measure $\mu$ on $G$ it follows that $0 < \mu(B(n,M)) < \infty$ since the balls $B(n,M)$ for $n \in N$ are precompact due to the properness of the metric. If we can show that $\mathcal{U}$ is a concatenation, then it follows that $\mathcal{U}$ is the uniform covering on $G$. \par 
Since $d$ restricts to a locally finite left-invariant metric on $N$ we have \[|B(n,R) \cap N| = |B(e,R) \cap N| < \infty\] for every $R \geq 0$. Assume that $B(n,M) \cap B(m,M) \ne \emptyset$ for $n,m \in N$. Then the triangle inequality implies that $m \in B(n,2M)$ and we have the bound \[N_\mathcal{U} \leq |B(e,2M) \cap N| < \infty,\] where $N_\mathcal{U}$ is the admissibility constant of the covering $\mathcal{U}$. Hence $\mathcal{U}$ is admissible and it is straightforward to see that $\mathcal{U}$ is a concatenation since \[B(e,kM) \subset B(e,M)^{k*}, \qquad \bigcup_{k = 1}^{\infty}B(e,kM) = G.\] \par 
By picking $n \in B(n,M)$ we conclude that $(N,d_\mathcal{U})$ is quasi-isometric to the uniform metric space $(G,d_{\mathcal{U}})$. Moreover, it is clear that $d_\mathcal{U}$ is a left-invariant metric on $N$ by construction. The result follows since the quasi-isometry class of a finitely generated group does not depend on the choice of a proper, left-invariant metric.
\end{proof}

Note that the uniform metric space $(G,d_{\mathcal{U}})$ is also quasi-isometric to $(G,d)$ since $N$ was a net in $G$. However, two left-invariant and compatible metrics on $G$ are not necessarily quasi-isometric. While the uniform metric space is quasi-isometric to $N$ with \textit{any} proper, left-invariant metric, this does not hold for $G$. \par
Although the number of assumptions in Theorem \ref{uniform_structure} might look overwhelming at first, there are many settings of this type. In particular, any left-invariant Riemannian metric on a connected Lie group $G$ induces a left-invariant and proper metric $d$ on $G$. Notice that any two left-invariant Riemannian metrics on a Lie group induce quasi-isometric distances. \par
It is important to keep in mind that an arbitrary locally compact group might not have a proper, left-invariant metric compatible with its topology. In fact, a classical result of Struble \cite{Metrics_on_locally_compact_groups} gives that the existence of a compatible, proper, and left-invariant metric on $G$ is equivalent to $G$ being second countable. Therefore, we restrict our attention to second countable and path-connected locally compact groups to avoid pathological examples.

\subsection{Stratified Lie Groups}
\label{sec: Carnot Groups}

We will now investigate a large class of examples within nilpotent Lie groups called \textit{stratified Lie groups}. In this setting, we will obtain stronger statements in Proposition \ref{Carnot_group_result} and Theorem \ref{growth_vector_result} than what was possible for general path-connected, locally compact groups. 

\begin{definition}
A \textit{stratified Lie group} $G$ is a connected and simply connected Lie group such that its Lie algebra $\mathfrak{g}$ has a \textit{stratification} \[\mathfrak{g} = V_1 \oplus \cdots \oplus V_s, \qquad [V_1, V_j] = V_{j+1}, \quad j=1, \dots, s-1, \qquad [V_1, V_s] = 0.\] The \textit{homogeneous dimension} of a stratified Lie algebra is defined to be \[Q := \sum_{j = 1}^{s}j \cdot \textrm{dim}_{\mathbb{R}}(V_j).\]
\end{definition}
The homogeneous dimension of a stratified Lie group is by definition the homogeneous dimension of its Lie algebra and is independent of the chosen stratification of the Lie algebra by \cite[Proposition 1.17]{Le_Donne_1}. The Lie group exponential map from $\mathfrak{g}$ to $G$ is a diffeomorphism for stratified Lie groups. Moreover, the Haar measure on $G$ is simply the push-forward of the Lebesgue measure on $\mathfrak{g}$ under the exponential map. In particular, every stratified Lie group is unimodular and diffeomorphic to Euclidean space. \par 
On stratified Lie groups there is a class of metrics that are intimately tied with the stratification of the Lie algebra: Fix an inner product $\langle \cdot, \cdot\rangle$ on $V_1$ and left translate this to obtain a Riemannian metric $g$ on $G$ that is only defined on the subbundle \[\mathcal{H} \subset TM, \qquad \mathcal{H}_x := dL_{x}V_1, \quad x\in G.\] The metric $g$ is called a \textit{sub-Riemannian metric} on $G$. An absolutely continuous curve $\gamma:[a,b] \to G$ is called \textit{horizontal} if \[dL_{\gamma(t)}^{-1}(\gamma(t)) \in V_1 \subset \mathfrak{g},\] for almost every $t \in [a,b]$. This gives a left-invariant distance function $d_{CC}$ by considering the infimum over horizontal curves: For $x,y \in G$ we define \[d_{CC}(x,y) := \inf_{\gamma}\int_{a}^{b}|\dot{\gamma}(t)|\, dt,\] where the infimum is taken over all horizontal curves such that $\gamma(a) = x$ and $\gamma(b) = y$. The distance function $d_{CC}$ is called the \textit{Carnot-Carath\'{e}odory distance} on $G$. The completeness of $(G,d_{CC})$ follows from Chow's Theorem \cite[Chapter 2]{Montgomery} in sub-Riemannian geometry. It is also common to refer to a stratified Lie group $G$ together with the data $(\mathcal{H},g)$ as a \textit{Carnot group} in the sub-Riemannian literature. \par 
Finally, recall that if $X_1, \dots, X_n$ is a basis for $\mathfrak{g}$ then the numbers $\{c_{ij}^{k}\}$ defined by \[[X_i,X_j] = \sum_{k = 1}^n c_{ij}^{k}X_k, \qquad i,j,k = 1, \dots, n,\] are called the \textit{structure constants} of the Lie algebra $\mathfrak{g}$ in the basis $X_1, \dots, X_n$. We call a Lie group \textit{realizable over the rationals} if there exists a basis for its Lie algebra such that the resulting structure constants are rational numbers.

\begin{proposition}
\label{Carnot_group_result}
Let $G$ be a stratified Lie group that is realizable over the rationals and let $N \subset G$ be any lattice in $G$. Then the uniform metric space $(G,d_\mathcal{U})$ is quasi-isometric to $(N,d)$, where $d$ is any proper, left-invariant metric on $N$.
\end{proposition}

\begin{proof}
Fix a stratification $\mathfrak{g}= V_1 \oplus \cdots \oplus V_s$ for the Lie algebra $\mathfrak{g}$ of $G$. The existence of a lattice $N$ in $G$ is equivalent to the requirement that $G$ is realizable over the rationals by \cite[Theorem 2.12]{Discrete_Subgroups}. Moreover, every lattice in a stratified Lie group is a finitely generated nilpotent group \cite[Theorem 2.10]{Discrete_Subgroups} that is \textit{uniform} \cite[Theorem 2.1]{Discrete_Subgroups}, that is, the quotient space $G/N$ is compact. \par 
Fix a Carnot-Carath\'eodory distance $d_{CC}$ on $G$ arising from an inner product on $V_1$ and notice that $N$ is then a net since we can write \[G = \bigcup_{n \in N}nC,\] where $C$ is some compact subset. Moreover, it follows from \cite[Corollary 5.5.9]{Geometric_Group_Theory} that the quasi-isometry class of $(N,d_{CC})$ does not depend on the choice of the lattice. The inclusion \[\overline{B_{d_{CC}}(e,R) \cap N} \subset \overline{B_{d_{CC}}(e,R)}\] together with the properness of $d_{CC}$ implies that $\overline{B_{d_{CC}}(e,R) \cap N}$ is finite due to the discreteness of $N$. Hence the metric $d_{CC}$ restricted to $N$ is locally finite and the result follows from Theorem \ref{uniform_structure}.
\end{proof}

\begin{remark}
Whenever the dimension of the Lie group is less than seven, the assumption that the Lie group is realizable over the rationals is automatically satisfied. This follows from the classification of real nilpotent Lie algebras with low dimension given in \cite{Lie_Algebras_Low_Dim}.
\end{remark}

Another useful invariant of a finitely generated group is its growth type. We will not go into the explicit definition of this since it slightly cumbersome and is well explained in \cite[Chapter 6]{Geometric_Group_Theory}. The idea is that the number of elements in $B(e,n)$ for a finitely generated group $N$ with proper, left-invariant metric is not a quasi-isometric invariant. However, the \textit{growth type} (e.g. if it grows linearly, quadratically, or exponentially) is a quasi-isometric invariant of the group. We will illustrate how this can be used in the following example. 

\begin{example}
\label{Heisenberg_example}
For $n \in \mathbb{N}$ we consider the Heisenberg group $(\mathbb{H}_{2n+1},*)$  consisting of all matrices on the form \[\left\{\begin{pmatrix} 1 & \mathbf{a} & c \\ 0 & I_{n \times n} & \mathbf{b} \\ 0 & 0 & 1\end{pmatrix} \, \Big| \, \mathbf{a}, \mathbf{b} \in \mathbb{R}^{n}, \,c \in \mathbb{R}\right\},\] where the operation $*$ denotes the usual matrix multiplication. It is a connected and simply connected Lie group whose Lie algebra $\mathfrak{g}_{2n+1}$ can be identified as a vector space with $\mathbb{R}^{2n+1} = \mathbb{R}^{2n} \oplus \mathbb{R}$. If $e_1, \dots, e_{2n+1}$ is the standard basis for $\mathbb{R}^{2n+1}$ then the Lie bracket satisfies \[[e_i,e_j] = \delta_{i + n,j}e_{2n+1}, \quad i \leq j < 2n+1, \qquad [e_i,e_{2n+1}] = 0.\] \par 
Fix an inner product $\langle \cdot, \cdot \rangle$ on $\mathbb{R}^{2n} \subset \mathfrak{g}_{2n+1}$ making the basis $e_1, \dots, e_{2n}$ orthonormal. We can equip $(\mathbb{H}_{2n+1},*)$ with a sub-Riemannian metric $g$ by left translating $\langle \cdot, \cdot \rangle$. The subset $\mathbb{Z}^{2n+1} \subset \mathbb{H}_{2n+1}$ is a finitely generated subgroup. The metric $d_{CC}$ restricts to a locally finite metric on $\mathbb{Z}^{2n+1}$ such that $\mathbb{Z}^{2n+1}$ is a net in $\mathbb{H}_{2n+1}$ due to the reasons pointed out in the proof of Proposition \ref{Carnot_group_result}. Hence by Theorem \ref{uniform_structure} it follows that the uniform metric space $(\mathbb{H}_{2n+1},d_{\mathcal{U}})$ is quasi-isometric to $(\mathbb{Z}^{2n+1},*)$ with any proper, left-invariant metric. A tedious but straightforward computation shows that the group $(\mathbb{Z}^{2n+1},*)$ has polynomial growth of order $2n+2$ while $(\mathbb{Z}^{k},+)$ has polynomial growth of order $k$. Since growth type is a quasi-isometric invariant by \cite[Corollary 6.2.6]{Geometric_Group_Theory} we have that \[(\mathbb{H}_{2n+1},d_{\mathcal{U}})\not \simeq (\mathbb{H}_{2m+1},d_{\mathcal{U}}), \, \, m \neq n, \qquad (\mathbb{H}_{2s+1},d_{\mathcal{U}})\not \simeq (\mathbb{R}^k,d_{\mathcal{U}}), \, \, k \neq 2s +2.\] However, since $(\mathbb{Z}^{2n+1},*)$ and $(\mathbb{Z}^{2n+2},+)$ have the same polynomial growth we need a different approach to show that $(\mathbb{H}_{2n+1},d_{\mathcal{U}})$ is not quasi-isometric to $(\mathbb{R}^{2n+2},d_{\mathcal{U}})$. \par 
Assume by contradiction that $(\mathbb{H}_{2n+1},d_{\mathcal{U}}) \simeq (\mathbb{R}^{2n+2},d_\mathcal{U})$. It follows from \cite[Corollary 6.3.16]{Geometric_Group_Theory} that $(\mathbb{Z}^{2n+1},*)$ then would have a finite index subgroup isomorphic to $(\mathbb{Z}^{2n+2},+)$. 
By intersecting all the conjugates of $(\mathbb{Z}^{2n+2},+)$ in $(\mathbb{Z}^{2n+1},*)$ one can assure that there exists a normal abelian subgroup of $(\mathbb{Z}^{2n+1},*)$ with finite index. The reason the intersection still has finite index is due to the easily verifiable formula \[|G:B\cap C| \leq |G:B| \cdot |G:C|,\] when $B,C$ are subgroups of $G$ with finite index. However, since $(\mathbb{Z}^{2n+1},*)$ is nilpotent and torsion free, it follows from \cite[Lemma 3.1]{Residual} that this forces $(\mathbb{Z}^{2n+1},*)$ to be abelian. Since this is not the case the claim follows.
\end{example}

\begin{remark}
The uniform covering $\mathcal{U}(\mathbb{H}_{3})$ is can be considered on $\mathbb{R}^3$ since $\mathbb{H}_3$ is diffeomorphic to $\mathbb{R}^3$. There, it is precisely the covering $\mathcal{P}$ introduced in Example \ref{heisenberg_3_example}. It thus follows from Example \ref{Heisenberg_example} that the two coverings $\mathcal{P}$ and $\mathcal{U}$ in Example \ref{heisenberg_3_example} are not equivalent coverings.
\end{remark}

Given a stratified Lie group $G$ with Lie algebra $\mathfrak{g} = V_1 \oplus \cdots \oplus V_s$, we call the multi-index \[\mathfrak{G}(G) := (n_{1}, \dots, n_{s}),\] the \textit{growth vector} of $G$, where $n_{i} := \textrm{dim}_{\mathbb{R}}(V_i)$ for $i = 1, \dots , s.$ The argument we used in the last part of Example \ref{Heisenberg_example} does not generalize easily. We remedy this by proving a stronger statement about when two uniform metric spaces on different stratified Lie groups can not be quasi-isometric.

\begin{theorem}
\label{growth_vector_result}
Let $G$ be a stratified Lie group and assume that $N \subset G$ is a lattice in $G$. Then $N$ has polynomial growth type of order equal to the homogeneous dimension of $G$. Let $H$ be another stratified Lie group that is realizable over the rationals such that the uniform metric spaces $(G,d_{\mathcal{U}})$ and $(H,d_{\mathcal{U}})$ are quasi-isometric. Then their growth vectors $\mathfrak{G}(G)$ and $\mathfrak{G}(H)$ have to be equal. 
\end{theorem}

\begin{proof}
We will build a correspondence between the lower central series of $N$ and the stratification on the Lie algebra $\mathfrak{g} = V_1 \oplus \dots \oplus V_s$. Consider the commutator subgroup $[G,G] \subset G$. Then \cite[Theorem 3.50]{Warner} implies that $[G,G]$ is a Lie subgroup of $G$ whose corresponding Lie algebra is isomorphic to $[\mathfrak{g},\mathfrak{g}] = V_2 \oplus \dots \oplus V_s$. Denote the projection onto the quotient by $\pi:G \to G/[G,G]$. \par 
It is straightforward to check that $G/[G,G]$ is isomorphic as a Lie group to Euclidean space and $\pi(N)$ is a lattice in $G/[G,G]$. However, lattices in Euclidean spaces are finitely generated abelian groups whose rank is equal to the dimension of the ambient Euclidean space. Hence it follows that $\pi(N)$ is generated by $\textrm{dim}(G/[G,G]) = \textrm{dim}(\mathfrak{g}/[\mathfrak{g},\mathfrak{g}]) = \textrm{dim}(V_1)$ elements. This gives \[\textrm{rank}_{\mathbb{Z}} C_{0}(N)/C_{1}(N) = \textrm{dim}_{\mathbb{R}}(V_1),\]
where $C_{i}(N)$ denotes the $i$'th term in the lower central series of $N$. We can proceed inductively to obtain that 
\begin{equation}
\label{correspondance_rank_dim}
    \textrm{rank}_{\mathbb{Z}} \left(C_{i}(N)/C_{i+1}(N)\right) = \textrm{dim}_{\mathbb{R}}(V_{i+1}), \quad i = 0 , \dots, s - 1.
\end{equation} 
The first statement of the theorem now follows from the Bass - Guivarc'h formula \cite[Theorem 2]{Bass}, stating that the polynomial growth of a finitely generated nilpotent group $N$ is precisely \[\sum_{k = 1}^{n} k \cdot \textrm{rank}_{\mathbb{Z}}\left(C_{k-1}(N)/C_{k}(N)\right).\] \par 
Let $H$ be another stratified Lie group that is realizable over the rationals and pick a lattice $M$ in $H$. A quasi-isometry between the uniform spaces on $G$ and $H$ induce a quasi-isometry between $(N,d_{N})$ and $(M,d_{M})$, where $d_{N}$ and $d_{M}$ are any proper, left-invariant metrics. Since the rank of the the quotients in the lower central series of a finitely generated nilpotent group are quasi-isometric invariants, we have that  \[\textrm{rank}_{\mathbb{Z}} \left(C_{i}(N)/C_{i+1}(N)\right) = \textrm{rank}_{\mathbb{Z}} \left(C_{i}(M)/C_{i+1}(M)\right).\] The correspondence (\ref{correspondance_rank_dim}) gives that the growth vectors $\mathfrak{G}(G)$ and $\mathfrak{G}(H)$ are the same. 
\end{proof}

\begin{example}
\label{Engel_group}
Let $\mathfrak{g}$ be the nilpotent Lie algebra spanned by the elements $X_1, X_2, X_3 , X_4$ with non-trivial bracket relations \[[X_1,X_2] = X_3, \quad [X_1,X_3] = X_4.\] We call $\mathfrak{g}$ the \textit{Engel algebra} and it is has a stratification given by \[\mathfrak{g} = \textrm{span}_{\mathbb{R}}\{X_1,X_2\} \oplus \textrm{span}_{\mathbb{R}}\{X_3\} \oplus \textrm{span}_{\mathbb{R}}\{X_4\}.\] The connected and simply connected Lie group $G$ corresponding to $\mathfrak{g}$ is called the \textit{Engel group} and appears for instance in \cite{Le_Donne_2}. Since $G$ is diffeomorphic to $\mathbb{R}^4$ through the exponential map, we can consider the two coverings $\mathcal{U}(G)$ and $\mathcal{U}\left(\mathbb{R}^4\right)$ on $\mathbb{R}^4$. \par 
To check that two coverings are not equivalent is not a complete triviality from a computational perspective. However, their uniform metric spaces are not quasi-isometric by Theorem \ref{growth_vector_result} since their growth vectors are different. Hence the coverings they induce on $\mathbb{R}^4$ are non-equivalent by Proposition \ref{generalization_of_equivalent_coverings}. This illustrates the novelty of the large scale approach, even when the coverings are on the same space. 
\end{example}

\subsection{More Examples}
\label{sec: More Examples}

\subsubsection{Solvable Groups}

We will now consider the more general class of solvable Lie groups and we begin by recalling the definition of an (abstract) solvable group. The \textit{derived series} of a group $N$ is defined by \[N^{(0)} := N, \quad N^{(i)} := [N^{(i-1)},N^{(i-1)}],\] for $i \geq 1$.  A group $N$ is said to be \textit{solvable} if its derived series eventually reaches the trivial group. Every nilpotent group is solvable, although the converse is false. A group $N$ is called \textit{virtually solvable} if it contains a solvable subgroup of finite index. \par 
To see that virtually solvable groups play a prominent role in the setting of uniform metric spaces on Lie groups, consider a connected Lie subgroup $G$ of $GL(n,\mathbb{R})$ for $n \geq 1$. Assume that $d$ is a proper, left-invariant metric on $G$ and that $N$ is a finitely generated subgroup of $G$ such that $d$ restricts to a locally finite metric on $N$. Then Theorem \ref{uniform_structure} shows that $(G,d_\mathcal{U}) \simeq (N,d_N)$, where $d_N$ is any proper, left-invariant metric on $N$. Since $N$ is a finitely generated subgroup of $GL(n,\mathbb{R})$ we can apply the famous Tits Alternative \cite[Theorem 4.4.7]{Geometric_Group_Theory} in group theory to conclude that $N$ is either virtually solvable or has a free subgroup of rank two as a finite index subgroup. Motivated by this, we examine the uniform metric spaces on solvable Lie groups more closely.

\begin{definition}
A \textit{solvable} Lie group is a connected Lie group such that its Lie algebra $\mathfrak{g}$ satisfies $\mathfrak{g}^n = \{0\}$ for some $n \in \mathbb{N}_{0}$, where \[\mathfrak{g}^{0} := \mathfrak{g}, \qquad \mathfrak{g}^i := [\mathfrak{g}^{i-1},\mathfrak{g}^{i-1}], \quad i \geq 1.\]
\end{definition}
An example of a solvable Lie group is all upper-triangular $n \times n$ matrices with positive determinant. As we will be interested in lattices in solvable Lie groups so that we can apply Theorem \ref{uniform_structure}, let us remark that the existence of lattices in solvable Lie groups are more complicated that in the nilpotent case. Unlike a stratified Lie group, a solvable Lie group does not need to be unimodular, that is, the right and left Haar measures might be different. There are no lattices in a non-unimodular locally compact group by \cite[Remark 1.9]{Discrete_Subgroups}. In particular, the \textit{affine group} (also known as the $Ax+b$ group) given by \[\textrm{Aff} := \left\{\begin{pmatrix} a & b \\ 0 & 1\end{pmatrix}\Big| \, a> 0, \, b \in \mathbb{R}\right\}\]does not admit lattices even though it is solvable. We will relate the uniform metric spaces on solvable Lie groups admitting lattices to the following subclass of finitely generated solvable groups.

\begin{definition}
A group $\Gamma$ is \textit{polycyclic} if it admits a chain of subgroups \[\Gamma = \Gamma_0 \supseteq \Gamma_1 \supseteq \dots \supseteq\Gamma_{k} = \{e \},\] where each term in the chain is a normal subgroup of the previous term and the quotients $\Gamma_{i-1}/\Gamma_i$ are cyclic groups for $i = 1, \dots, k$. It is called \textit{strongly polycyclic} if it admits such a chain where each quotient $\Gamma_{i-1}/\Gamma_i$ is infinitely cyclic. 
\end{definition}

\begin{proposition}
\label{solvable_group_result}
Let $G$ be a connected and simply connected solvable Lie group and assume there exists a lattice $\Gamma$ in $G$. Then the uniform metric space $(G,d_{\mathcal{U}})$ is quasi-isometric to $(\Gamma,d)$ where $d$ is any proper, left-invariant metric on $\Gamma$. Moreover, $\Gamma$ is strongly polycyclic.
\end{proposition}

\begin{proof}
Any lattice in a solvable Lie group is uniform by \cite[Theorem 3.1]{Discrete_Subgroups}. It follows from \cite[Proposition 3.7]{Discrete_Subgroups} that any lattice in a simply connected solvable Lie group is strongly polycyclic and hence finitely generated. By fixing a Riemannian metric $g$ on $G$ by left translating an inner product on the Lie algebra, it is clear that all the conditions in Theorem \ref{uniform_structure} are satisfied and the result follows. 
\end{proof}

\subsubsection{The Special Linear Group and the Hyperbolic Plane}
\label{sec: The Special Linear Group and the Hyperbolic Plane}

We will illustrate a uniform metric space that has fundamentally different properties than those built on solvable Lie groups. Consider the Lie group $SL(2,\mathbb{R})$ of $2 \times 2$ matrices with real coefficients and unit determinant. It is related to the hyperbolic plane $\mathbb{H}^2$ with the usual hyperbolic distance by the fact that $SL(2,\mathbb{R})$ acts on $\mathbb{H}^2$ by M\"{o}bius transformations \[\begin{pmatrix} a & b \\ c & d\end{pmatrix} \cdot z := \frac{az + b}{cz + d}, \qquad A = \begin{pmatrix} a & b \\ c & d\end{pmatrix} \in SL(2,\mathbb{R}), \, z \in \mathbb{H}^2.\] Notice that both $A$ and $-A$ induce the same transformation. \par 
An action of a discrete group $G$ on a topological space $X$ is said to be \textit{properly discontinuous} if every point $x \in X$ has a neighbourhood $U$ such that $g \cdot U \cap U = \emptyset$ for every non-identity element $g \in G$. Finally, recall that a group action is said to be free if $g \cdot x = x$ for some $x \in X$ and $g \in G$ implies that $g$ is the identity element of the group $G$. 

\begin{theorem}
\label{SL}
The uniform metric space $(SL(2,\mathbb{R}),d_\mathcal{U})$ is quasi-isometric to the fundamental group of any compact Riemann surface of genus $g \geq 2$. Moreover, this is again quasi-isometric to the hyperbolic space $\mathbb{H}^2$ with its usual hyperbolic distance. In particular, the uniform metric space $(SL(2,\mathbb{R}),d_\mathcal{U})$ is quasi-hyperbolic.
\end{theorem}

\begin{proof}
Fix an inner product $\langle \cdot, \cdot \rangle$ on the Lie algebra $\mathfrak{sl}(2,\mathbb{R})$ of $SL(2,\mathbb{R})$ consisting of $2 \times 2$ matrices with real coefficients and zero trace. Left translate this to obtain a Riemannian metric on $SL(2,\mathbb{R})$ and consider the Carnot-Carath\'{e}odory metric $d_{CC}$ associated to it. Then $(SL(2,\mathbb{R}),d_{CC})$ satisfies all the initial assumptions in Theorem \ref{uniform_structure}.  \par
Let $X$ be a compact Riemann surface of genus $g \geq 2$. The fundamental group $\pi_{1}(X)$ of $X$ can be realized as a uniform and torsion free discrete subgroup $\Gamma$ of $SL(2,\mathbb{R})$. Conversely, any uniform and torsion free discrete subgroup $\Gamma$ of $SL(2,\mathbb{R})$ acts on $\mathbb{H}^2$ freely and properly discontinuously such that the orbit space $\mathbb{H}^2/\Gamma$ is a compact Riemann surface. These observations are are built up of several standard results about compact Riemannian surfaces and they can all be found in the lecture notes \cite{Fuchsian}. \par 
Fix such a uniform and torsion free discrete subgroup $\Gamma$ of $SL(2,\mathbb{R})$. Then $\Gamma$ acts on $SL(2,\mathbb{R})$ by left translations and it follows from the Milnor-\v{S}varc lemma \cite[Proposition 1.3.13]{Large_Scale} that $\Gamma$ is finitely generated. The fact that $\Gamma$ is uniform implies that it is a net in $(SL(2,\mathbb{R}),d_{CC})$. The discreteness of $\Gamma$ implies that $d_{CC}$ is locally finite on $\Gamma$. We can conclude by Theorem \ref{uniform_structure} that the uniform metric space $(SL(2,\mathbb{R}),d_\mathcal{U})$ is quasi-isometric to $\Gamma$ with any proper, left-invariant metric. The choice of $\Gamma$ does not matter since \cite[Corollary 5.5.9]{Geometric_Group_Theory} implies that any two uniform, discrete subgroups of $SL(2,\mathbb{R})$ are quasi-isometric. The quasi-isometry between the fundamental group $\pi_{1}(X)$ and the hyperbolic plane $\mathbb{H}^2$ is well-known and can be found in \cite[Corollary 5.4.10]{Geometric_Group_Theory}. The final statement follows from Lemma \ref{hyperbolicity_lemma} (b).
\end{proof}

\begin{remark}
In the proof of Theorem \ref{SL} it is tempting to consider the lattice $SL(2,\mathbb{Z})$ in $SL(2,\mathbb{R})$ instead of $\Gamma$. However, $SL(2,\mathbb{Z})$ has a free group of rank two as a finite index subgroup as showed in \cite[Example 4.4.1]{Geometric_Group_Theory}. This implies together with \cite[Theorem 1]{Hyperbolic_Plane_Virtually_Free} that $SL(2,\mathbb{Z})$ is not quasi-isometric to $\mathbb{H}^2$. The reason for this failure lies with the non-compactness of the homogeneous space $SL(2,\mathbb{R})/SL(2,\mathbb{Z})$.
\end{remark}

\begin{proposition}
\label{SL_embeddings}
The uniform metric space $(SL(2,\mathbb{R}),d_\mathcal{U})$ is not quasi-isometric to $(\mathbb{H}_{2n+1},d_\mathcal{U})$ or $(\mathbb{R}^k,d_\mathcal{U})$ for any $k,s \in \mathbb{N}$. In fact, there are no quasi-isometric embeddings \[(\mathbb{R}^k,d_\mathcal{U}) \longrightarrow (SL(2,\mathbb{R}),d_\mathcal{U}), \qquad (\mathbb{H}_{2n+1},d_\mathcal{U}) \longrightarrow (SL(2,\mathbb{R}),d_\mathcal{U}),\] unless $k = 1$. 
\end{proposition}

\begin{proof}
Consider the elements \[A = \begin{pmatrix} 1 & e_1 & 0 \\ 0 & I_{n\times n} & 0 \\ 0 & 0 & 1 \end{pmatrix}, \, B = \begin{pmatrix} 1 & 0 & 1 \\ 0 & I_{n\times n} & 0 \\ 0 & 0 & 1 \end{pmatrix} \in \mathbb{H}_{2n+1},\] where $e_1 = (1, 0, \dots, 0)$. The subgroup $\langle A, B \rangle$ generated by $A$ and $B$ is commutative and the mapping 
\begin{align*}
\phi: \langle A,B \rangle & \longrightarrow \mathbb{Z}^2 \\
A^{r}B^{s} & \longmapsto (r,s)
\end{align*}
gives an isomorphism between $\langle A,B \rangle$ and $\mathbb{Z}^2$. Hence it follows from Lemma \ref{hyperbolicity_lemma} (d) that the Heisenberg groups are not hyperbolic. We mentioned in Lemma \ref{hyperbolicity_lemma} (c) that $\mathbb{Z}^k$ is not a hyperbolic group unless $k = 1$. Hence neither of the quasi-isometric embeddings $(\mathbb{R}^k,d_\mathcal{U}) \longrightarrow (SL(2,\mathbb{R}),d_\mathcal{U})$ or $(\mathbb{H}_{2n+1},d_\mathcal{U}) \longrightarrow (SL(2,\mathbb{R}),d_\mathcal{U})$ are possible due to Lemma \ref{hyperbolicity_lemma} (a) for $n \in N$ and $k \geq 2$. \par 
For $k = 1$ one obtain several quasi-isometric embeddings $(\mathbb{R},d_\mathcal{U}) \longrightarrow (SL(2,\mathbb{R}),d_\mathcal{U})$ from Lemma \ref{hyperbolicity_lemma} (c). We can not use hyperbolicity to conclude that $(\mathbb{R},d_{\mathcal{\mathcal{U}}})$ is not quasi-isometric to $(SL(2,\mathbb{R}),d_{\mathcal{\mathcal{U}}})$. However, we can consider their asymptotic dimensions together with Theorem \ref{SL} to derive \[\textrm{asdim}(SL(2,\mathbb{R}),d_\mathcal{U}) = \textrm{asdim}(\mathbb{H}^2) = 2 \neq 1 = \textrm{asdim}(\mathbb{Z}) = \textrm{asdim}(\mathbb{R},d_\mathcal{U}).\]
Here we have used that the asymptotic dimension of $\mathbb{H}^2$ is equal to two, a result going back to Gromov \cite{Gromov}. Hence the claim follows from the quasi-isometric invariance of asymptotic dimension.
\end{proof}

Notice that we used both asymptotic dimension and hyperbolicity in the proof of Proposition \ref{SL_embeddings}. Arguments such as these are our main motivation for considering invariants from large scale geometry. For another class of examples, we refer the reader interested in shearlet groups to the Ph.D. thesis of Ren\'{e} Koch \cite[Section 5.4]{ThesisRene} where novel results regarding the relationship between admissible groups, dual orbits and quasi-isometries are proved.

\section{Decomposition Spaces and Geometric Embeddings}
\label{Chapter_Decomposition_Spaces}

This section is devoted to introducing embeddings between decomposition spaces that induce quasi-isometric embeddings between the underlying coverings called \textit{geometric embeddings}. In Subsection \ref{sec: Spatially_Implemented_Geometric_Embeddings} we will give some criteria for when quasi-isometries between the underlying coverings can induce geometric embeddings between decomposition spaces.

\subsection{Definitions and Basic Properties}
\label{sec:Definitions_and_Basic_Properties}

We will start by reviewing basic definitions and results regarding decomposition spaces given in \cite{Hans_Grobner}. This is done to make our exposition complete as well as to fix notation and settle our conventions. Throughout this section, we let $X$ denote an arbitrary locally compact topological space and denote by $(\mathcal{A},\|\cdot\|_{\mathcal{A}})$ a subspace of $C_{b}(X,\mathbb{C})$ with a norm $\|\cdot\|_{\mathcal{A}}$ making it into a Banach algebra under pointwise multiplication. Moreover, we additionally stipulate that $(\mathcal{A},\|\cdot\|_{\mathcal{A}})$ is closed under complex conjugation and that it is \textit{regular}, that is, $(\mathcal{A},\|\cdot\|_{\mathcal{A}})$ is sufficiently large to separate points from closed sets by continuous functions. A \textit{partition of unity} $\Phi = (\varphi_{i})_{i \in I}$ on $X$ subordinate to an admissible covering $\mathcal{Q} = (Q_i)_{i \in I}$ is a collection of non-negative continuous functions such that 
\begin{equation}
\label{sum_partition_of_unity}
    \textrm{supp}(\varphi_i) \subset Q_i, \quad \sum_{i \in I}\varphi_{i} \equiv 1.
\end{equation}
Since the covering $\mathcal{Q}$ is assumed to be admissible, there is no convergence issue in the sum (\ref{sum_partition_of_unity}).

\begin{definition}
Let $\mathcal{Q} = (Q_i)_{i \in I}$ be an admissible covering on $X$. A \textit{bounded admissible partition of unity} (BAPU) in $\mathcal{A}$ subordinate to $\mathcal{Q}$ is a partition of unity $\Phi = (\varphi_{i})_{i \in I}$ subordinate to $\mathcal{Q}$ where $\varphi_{i} \in \mathcal{A}$ for every $i \in I$ and
\begin{equation}
\label{BAPU-bound}
    \sup_{i \in I}\|\varphi_i\|_{\mathcal{A}} < \infty.
\end{equation}
It is common to refer to $\Phi$ as a $\mathcal{Q}$-BAPU to emphasize the covering $\mathcal{Q}$ in question. 
\end{definition}

We denote by $\mathcal{A}_0$ the elements of $\mathcal{A}$ that have compact support. When forming the decomposition space $\mathcal{D}(\mathcal{Q},B,Y)$ in Definition \ref{decomposition_spaces_definition}, we need some weak assumptions on the Banach spaces $(B,\|\cdot\|_{B})$ and $(Y,\|\cdot\|_{Y})$ to deduce nice properties of the decomposition spaces $\mathcal{D}(\mathcal{Q},B,Y)$. \par 
Our standing assumptions are that $B$ is continuously embedded into the dual $\mathcal{A}_{0}^*$, contains $\mathcal{A}_0$ as a dense subspace, and is a Banach module over $\mathcal{A}$ under pointwise operations. We assume that $(Y,\|\cdot\|_{Y})$ is a Banach space consisting of sequences on the index set $I$. Moreover, the finitely supported sequences are required to form a dense subspace of $Y$. Define the \textit{clustering map} $\Gamma_{\mathcal{Q}}:Y \longrightarrow Y$ by \[ (a_i)_{i \in I} \longmapsto \left(\sum_{j \in i^{*}}a_j\right)_{i \in I}.\]
It will henceforth be assumed that the clustering map $\Gamma_{\mathcal{Q}}$ is well-defined and bounded on $Y$.
Finally, we impose that $Y$ should be \textit{solid}, meaning that if $x = (x_i)_{i \in I}$ is a sequence in $Y$ and $y = (y_i)_{i \in I}$ is a sequence in $\mathbb{C}^{I}$ such that $|y_i| \leq |x_i|$ for every $i \in I$, then $y \in Y$ with $\|y\|_{Y} \leq \|x\|_{Y}$. We refer the reader to \cite[Section 2]{Hans_Grobner} for a more thorough discussion of these assumptions.

\begin{definition}
\label{decomposition_spaces_definition}
Let $B$ and $Y$ be Banach spaces satisfying the standing assumptions above. Moreover, let $\Phi = (\varphi_i)_{i \in I}$ be a $\mathcal{Q}$-BAPU in $\mathcal{A}$ corresponding to an admissible covering $\mathcal{Q}$ on $X$. The \textit{decomposition (function) space} $\mathcal{D}(\mathcal{Q},B,Y)$ consists of all elements $f \in \mathcal{A}_{0}^{*}$ such that 
\begin{equation}
\label{decomposition_norm}
    \Big\|\left(\|f \cdot \varphi_i\|_{B}\right)_{i \in I}\Big\|_Y < \infty.
\end{equation}
We call $B$ the \textit{local component} and $Y$ the \textit{global component} of the decomposition space $\mathcal{D}(\mathcal{Q},B,Y)$.
\end{definition}

By equipping $\mathcal{D}(\mathcal{Q},B,Y)$ with the norm given by (\ref{decomposition_norm}) we have that $\mathcal{D}(\mathcal{Q},B,Y)$ is a Banach space by \cite[Theorem 2.2 A]{Hans_Grobner}. The observant reader will have noticed that we have excluded the $\mathcal{Q}$-BAPU $\Phi = (\varphi_i)_{i \in I}$ from the notation $\mathcal{D}(\mathcal{Q},B,Y)$. This is because \cite[Theorem 2.3]{Hans_Grobner} implies that different $\mathcal{Q}$-BAPU's give rise to equivalent norms. We summarize some well known properties of decomposition spaces in Proposition \ref{basic_properties} below. The last statement of Proposition \ref{basic_properties} is a straightforward extension of \cite[Corollary 2.6]{Hans_Grobner}.

\begin{proposition}
\label{basic_properties}
The (continuous) dual space of $\mathcal{D}(\mathcal{Q},B,Y)$ can be identified with the decomposition space $\mathcal{D}(\mathcal{Q},B^{*},Y^{*})$. In particular, reflexivity of the local and global components gives reflexivity of the corresponding decomposition space. Moreover, we have the norm convergence \[f = \sum_{i \in I}f\cdot \varphi_{i},\] in $\mathcal{D}(\mathcal{Q},B,Y)$ where $\Phi = (\varphi_i)_{i \in I}$ is any $\mathcal{Q}$-BAPU for $\mathcal{D}(\mathcal{Q},B,Y)$. Finally, a function $f$ belongs to $\mathcal{D}(\mathcal{Q},B,Y)$ if and only if there exist $k \in \mathbb{N}$ and $f_i \in B$ with $\textrm{supp}(f_i) \subset Q_{i}^{k*}$ such that $\{\|f_{i}\|_{B}\}_{i \in I} \in Y$ and $f = \sum_{i \in I}f_i$ in $\mathcal{A}_{0}^*$.
\end{proposition}

\begin{remark}
The requirement that $\mathcal{A}_{0}$ is dense in $B$ is only needed for the duality statement in Proposition \ref{basic_properties}, while the requirement that the finite sequences are dense in $Y$ is required for both the duality statement and the norm convergence statement in Proposition \ref{basic_properties}. The reader interested in cases where these requirements does not hold, such as $Y = l^{\infty}(I)$, can safely use all subsequent results that does not invoke these properties. 
\end{remark}

Many of the decomposition spaces appearing in the literature such as modulation spaces and Besov spaces are built on open subsets of some Euclidean space. However, they are not precisely decomposition spaces as defined in \cite{Hans_Grobner}, but rather a variation that incorporates the Fourier transform. We briefly outline this distinction and refer the reader to the paper \cite{Felix_main} for more details. \par
Let $\mathcal{Q} = (Q_i)_{i \in I}$ be an admissible covering for the open set $\emptyset \neq \mathcal{O} \subset \mathbb{R}^k$ with a $\mathcal{Q}$-BAPU $\Phi = (\varphi_i)_{i \in I}$. Moreover, let $B$ and $Y$ be Banach spaces satisfying the standing assumptions where $\mathcal{A} = \mathcal{F}L^1$ is the Fourier transform of all integrable functions. Then the decomposition space $\mathcal{D}(\mathcal{Q},\mathcal{F}L^p,Y)$ consists of all elements $f \in \mathcal{A}_{0}^{*}$ such that 
\begin{equation}
\label{Fourier_decomposition_norm}
    \Big\|\left(\Big\|f \cdot \varphi_i\Big\|_{\mathcal{F}L^{p}}\right)_{i \in I}\Big\|_Y = \Big\|\left(\Big\|\mathcal{F}^{-1}\left(f \cdot \varphi_i\right)\Big\|_{L^p}\right)_{i \in I}\Big\|_Y < \infty.
\end{equation}
The local component $\mathcal{F}L^{p}$ is a Banach module under pointwise multiplication over $\mathcal{A}$ since \[\|f \cdot a\|_{\mathcal{F}L^{p}} = \|\mathcal{F}^{-1}\left(f \cdot a\right)\|_{L^p} = \|\mathcal{F}^{-1}(f) \, * \, \mathcal{F}^{-1}(a)\|_{L^p} \leq \|f
\|_{\mathcal{F}L^p} \cdot \|a\|_{\mathcal{A}},\] for $a \in \mathcal{A}$ and $f \in \mathcal{F}L^{p}$. The expression (\ref{Fourier_decomposition_norm}) is well-defined by the uniform bound (\ref{BAPU-bound}).
\begin{definition}
The $\mathcal{F}$\textit{-type decomposition space} $\mathcal{D}^{{\mathcal{F}}}(\mathcal{Q},L^p,Y)$ is defined by \[\mathcal{D}^{{\mathcal{F}}}(\mathcal{Q},L^p,Y) := \mathcal{F}^{-1}\left(\mathcal{D}(\mathcal{Q},\mathcal{F}L^p,Y)\right).\] Hence $f\in \mathcal{D}^{{\mathcal{F}}}(\mathcal{Q},L^p,Y)$ if and only if $\mathcal{F}(f) \in \mathcal{A}_{0}^{*}$ and  \[\Big\|\left(\Big\|\mathcal{F}^{-1}\left(\mathcal{F}(f) \cdot \varphi_i\right)\Big\|_{L^{p}}\right)_{i \in I}\Big\|_Y < \infty. \]

\end{definition}
If we want to indicate that a decomposition space in a statement can be either a $\mathcal{F}$-type decomposition space or a standard decomposition space, we refer to it as a ($\mathcal{F}$-\textit{type}) \textit{decomposition space}.

\begin{remark}
One avenue that we have not pursued is to consider the quasi-Banach setting, that is, where the local component $(B,\|\cdot\|_{B})$ and the global component $(Y,\|\cdot\|_{Y})$ of ($\mathcal{F}$-type) decomposition spaces are quasi-Banach spaces. The most common examples are $B = L^p$ and $Y = l^q$ for $0 < p,q < 1$. Although these have received increased interest in the last few years, we will avoid this more technical case since the underlying geometry of the coverings are not affected by this extension. We refer the interested reader to \cite{Felix_main} for the most comprehensive exposition on decomposition spaces with quasi-Banach spaces as local and global components.
\end{remark}

\subsection{Geometric Embeddings}
\label{sec: Geometric Embeddings}

We now take up the question of whether one ($\mathcal{F}$-type) decomposition space embeds nicely into another ($\mathcal{F}$-type) decomposition space. As the ($\mathcal{F}$-type) decomposition spaces are Banach spaces, they can embed into each other as \textit{Banach spaces} without this actually reflecting the underlying geometry of the coverings. Moreover, the embedding may then be artificial and not readily available. Hence we will consider a refined notion of embeddings between ($\mathcal{F}$-type) decomposition spaces that incorporates the underlying coverings. \par

Recall that an \textit{embedding} between Banach spaces $(B_1,\|\cdot\|_{B_1})$ and $(B_2,\|\cdot\|_{B_2})$ is an injective linear map $F:B_1 \to B_2$ such that $\|F(f)\|_{B_2} \leq A\|f\|_{B_1}$ for some constant $A > 0$ not depending on $f \in B_1$. Let $(X,\mathcal{Q})$ be a covered space and consider a decomposition space $\mathcal{D}(\mathcal{Q},B,Y)$. We define the \textit{adapted support} of an element $f \in \mathcal{D}(\mathcal{Q},B,Y)$ with respect to the $\mathcal{Q}$-BAPU $\Phi = (\varphi_i)_{i \in I}$ to be \[\mathcal{C}[f] := \bigcup_{i \in I} \left\{Q_i \, \, \Big | \, \, \|f \cdot \varphi_{i}\|_{B} \neq 0 \right\}.\] Notice that $f_i := \sum_{j \in i^{*}}\varphi_{i} \in \mathcal{D}(\mathcal{Q},B,Y)$ is a non-zero function satisfying $\mathcal{C}[f_i] \subset Q_{i}^{2*}$.  If we are considering $\mathcal{F}$-type decomposition spaces, then the \textit{adapted support} of $f \in \mathcal{D}^{\mathcal{F}}(\mathcal{Q},L^{p},Y)$ with respect to the $\mathcal{Q}$-BAPU $\Phi = (\varphi_i)_{i \in I}$ is defined to be \[\mathcal{C}[f] := \bigcup_{i \in I} \left\{Q_i \, \, \Big | \, \, \|\mathcal{F}(f) \cdot \varphi_{i}\|_{\mathcal{F}L^{p}} \neq 0 \right\}.\]
Notice that the adapted support might depend on the choice of $\mathcal{Q}$-BAPU. However, it will be clear when we use the adapted support in Definition \ref{geometric_embedding_definition} below that the choice of $\mathcal{Q}$-BAPU is irrelevant.

\begin{definition}
\label{geometric_embedding_definition}
Let $\mathcal{D}(\mathcal{Q},B_1,Y_1)$ and $\mathcal{D}(\mathcal{P},B_2,Y_2)$ be ($\mathcal{F}$-type) decomposition spaces with underlying covered spaces $(X,\mathcal{Q})$ and $(Z,\mathcal{P})$. We say that a map $F:\mathcal{D}(\mathcal{Q},B_1,Y_1) \to \mathcal{D}(\mathcal{P},B_2,Y_2)$ is a \textit{geometric embedding} of decomposition spaces if it is an embedding of Banach spaces with the following additional requirement: There should exists constants $L,C > 0$ such that for any $k \in \mathbb{N}_{0}$ and any $f,g \in \mathcal{D}(\mathcal{Q},B_1,Y_1)$ with $\mathcal{C}[f] \subset Q_{i}^{k*}$ and $\mathcal{C}[g] \subset Q_{j}^{k*}$, we have
\begin{equation}
\label{geometric_condition}
 \frac{1}{L}d_{\mathcal{Q}}(x,y) - C \leq d_{\mathcal{P}}\left(z,w\right) \leq Ld_{\mathcal{Q}}(x,y) + C,
\end{equation}
where $x \in Q_{i}^{k*}$, $y \in Q_{j}^{k*}$, $z \in \mathcal{C}[F(f)]$ and $w \in \mathcal{C}[F(g)]$ are arbitrary. Two decomposition spaces $\mathcal{D}(\mathcal{Q},B_1,Y_1)$ and $\mathcal{D}(\mathcal{P},B_2,Y_2)$ are said to be \textit{geometrically isomorphic} if there exists an invertible geometric embedding from $\mathcal{D}(\mathcal{Q},B_1,Y_1)$ to $\mathcal{D}(\mathcal{P},B_2,Y_2)$ whose inverse is also a geometric embedding. 
\end{definition}
Although it would seem more convenient to require \eqref{geometric_condition} only for $f = \chi_{Q_i}$ and $g = \chi_{Q_j}$, this is often not sufficient for the simple reason that $\chi_{Q_i}$ might not be in $\mathcal{D}(\mathcal{Q},B_1,Y_1)$. An example where this happens is the modulation space $M^{1}(\mathbb{R}^n)$ defined in Subsection \ref{sec: Euclidean_Modulation_Spaces} since every element in $M^{1}(\mathbb{R}^n)$ is continuous. Moreover, we will give an example at the end of Subsection \ref{sec: David} showing that two decomposition spaces $\mathcal{D}(\mathcal{Q},B_1,Y_1)$ and $\mathcal{D}(\mathcal{P}, B_2, Y_2)$ can be equal as Banach spaces without being geometrically isomorphic. \par
To see why the definition of geometric embeddings encodes the geometry of the decomposition space, we consider the case where $X = Z$. Assume that the identity mapping \[\mathcal{D}(\mathcal{Q},B_1,Y_1) \ni f \longmapsto f \in \mathcal{D}(\mathcal{P},B_2,Y_2)\] is a geometric isomorphism. Then the identity map from $(X,d_{\mathcal{Q}})$ to $(X,d_{\mathcal{P}})$ is a quasi-isometry by (\ref{geometric_condition}). Hence it follows from Proposition \ref{generalization_of_equivalent_coverings} that the coverings $\mathcal{Q}$ and $\mathcal{P}$ are equivalent. Conversely, if the identity map $f:(X,d_{\mathcal{Q}}) \to (X,d_{\mathcal{P}})$ is a quasi-isometry, then the identity map acting on functions $f:X \to \mathbb{C}$ satisfies the estimate (\ref{geometric_condition}). However, it is not guaranteed that the identity $f \mapsto f$ maps $\mathcal{D}(\mathcal{Q},B_1,Y_1)$ continuously into $\mathcal{D}(\mathcal{P},B_2,Y_2)$. 

\begin{proposition}
\label{covering_invariance}
Let $\mathcal{D}(\mathcal{Q},B_1,Y_1)$ and $\mathcal{D}(\mathcal{P},B_2,Y_2)$ be ($
\mathcal{F}$-type) decomposition spaces with underlying covered spaces $(X,\mathcal{Q})$ and $(Z,\mathcal{P})$. If $F:\mathcal{D}(\mathcal{Q},B_1,Y_1) \to \mathcal{D}(\mathcal{P},B_2,Y_2)$ is a geometric embedding, then $F$ induces a quasi-isometric embedding between the metric spaces $(X,d_{\mathcal{Q}})$ and $(Z,d_\mathcal{P})$. In particular, the decomposition spaces $\mathcal{D}(\mathcal{Q},B_1,Y_1)$ and $\mathcal{D}(\mathcal{P},B_2,Y_2)$ can be geometrically isomorphic only when the associated metric spaces $(X,d_{\mathcal{Q}})$ and $(Z,d_{\mathcal{P}})$ are quasi-isometric.
\end{proposition}

\begin{proof}
Assume that $F:\mathcal{D}(\mathcal{Q},B_1,Y_1) \to \mathcal{D}(\mathcal{P},B_2,Y_2)$ is a geometric embedding. We define a map $\eta:(X,d_{\mathcal{Q}}) \to (Z,d_{\mathcal{P}})$ as follows:  For $x \in X$ we have $x \in Q_i$ for some $i \in I$. Choose a non-zero function $f \in \mathcal{D}(\mathcal{Q},B_1,Y_1)$ with $\mathcal{C}[f] \subset Q_{i}^{k*}$ for some $k \in \mathbb{N}_{0}$. Since $F$ is injective there exists an element $y \in \mathcal{C}\left[F(f)\right]$. Define $\eta(x) = y$. The estimate (\ref{geometric_condition}) gives that $\eta$ is a quasi-isometric embedding.
\end{proof}

We can now use results we have developed for covered spaces to deduce obstructions about geometric embeddings between ($\mathcal{F}$-type) decomposition spaces. Whenever we consider the uniform covering $\mathcal{U}(G)$ on a path-connected, locally compact group $G$, we use the simplified notation \[\mathcal{D}(G,B,Y) := \mathcal{D}(\mathcal{U}(G),B,Y), \qquad \mathcal{D}^{\mathcal{F}}(G,B,Y) := \mathcal{D}^{\mathcal{F}}(\mathcal{U}(G),B,Y).\]

\begin{proposition}
\label{restrictions_on_embeddings}
There are no geometric embeddings
\begin{align*}
    \mathcal{D}(\mathbb{R}^{k},B_1,Y_1) & \longrightarrow  \mathcal{D}(\mathbb{R}^{l},B_2,Y_2), \quad \, \, \, \, \, \, \, \, \, \, l < k, \\
    \mathcal{D}(\mathbb{H}_{2m+1},B_3,Y_3) & \longrightarrow  \mathcal{D}(\mathbb{H}_{2n+1},B_4,Y_4), \quad n < m, \\
    \mathcal{D}(\mathbb{R}^{k},B_1,Y_1) & \longrightarrow  \mathcal{D}(\mathbb{H}_{2n+1},B_4,Y_4), \quad 2n+1 < k, \\ 
    \mathcal{D}(\mathbb{H}_{2m+1},B_3,Y_3) & \longrightarrow  \mathcal{D}(\mathbb{R}^{l},B_2,Y_2), \quad \, \, \, \, \, \, \, \, \, \, l < 2m+1,
\end{align*}
where $B_1, \dots, B_4$ and $Y_1, \dots, Y_4$ are arbitrary Banach spaces satisfying the standing assumptions. Moreover, the decomposition spaces $\mathcal{D}(\mathbb{R}^{k},B_1,Y_1)$ and $\mathcal{D}(\mathbb{H}_{2n+1},B_4,Y_4)$ are not geometrically isomorphic for any $n,k \in \mathbb{N}$.
\end{proposition}

\begin{proof}
It follows from Proposition \ref{covering_invariance} that it suffices to show that there are no quasi-isometric embeddings between the underlying uniform metric spaces. 
Assume by contradiction that there exists a quasi-isometric embedding $\phi:(\mathbb{H}_{2m+1},d_\mathcal{U}) \to (\mathbb{H}_{2n+1},d_\mathcal{U})$ when $n < m$. Then \[\textrm{asdim}(\mathbb{H}_{2m+1},d_\mathcal{U}) \leq \textrm{asdim}(\mathbb{H}_{2n+1},d_\mathcal{U}).\] However, this contradicts \cite[Theorem 3.5]{Homological_Coherence} stating that the asymptotic dimension of the net $(\mathbb{Z}^{2n+1},*)$ in $\mathbb{H}_{2n+1}$ is equal to $2n+1$. Since we know that \[\textrm{asdim}(\mathbb{R}^k,d_{\mathcal{U}}) = \textrm{asdim}(\mathbb{Z}^k,+) = k,\] the other statements follows as well. Finally, the last claim follows from Example \ref{Heisenberg_example}.
\end{proof}

\begin{remark}
\label{remark_geometric_restrictions}
Since $\mathbb{H}_{2m+1}$ is diffeomorphic to $\mathbb{R}^{2m+1}$ we can consider the uniform covering $\mathcal{U}$ on $\mathbb{H}_{2m+1}$ as a covering on $\mathbb{R}^{2m+1}$. Hence $\mathcal{D}^{\mathcal{F}}(\mathbb{H}_{2m+1},L^q,Y_2)$ is well-defined. The statements in Proposition \ref{restrictions_on_embeddings} also hold if we consider the $\mathcal{F}$-type decomposition spaces $\mathcal{D}^{\mathcal{F}}(\mathbb{R}^{k},L^p,Y_1)$ and $\mathcal{D}^{\mathcal{F}}(\mathbb{H}_{2m+1},L^q,Y_2)$ for $1 \leq p,q < \infty$.
\end{remark}

\subsection{Spatially Implemented Geometric Embeddings}
\label{sec: Spatially_Implemented_Geometric_Embeddings}

In Proposition \ref{covering_invariance} we showed that geometric embeddings $F:\mathcal{D}(\mathcal{Q},B_1,Y_1) \to \mathcal{D}(\mathcal{P},B_2,Y_2)$ between ($\mathcal{F}$-type) decomposition spaces induce quasi-isometric embeddings between the associated metric spaces $(X,d_{\mathcal{Q}})$ and $(Z,d_{\mathcal{P}})$ of the underlying coverings. A question that naturally arises is whether the opposite might be true in certain situations: Does a quasi-isometric embedding between $(X,d_{\mathcal{Q}})$ and $(Z,d_{\mathcal{P}})$ induce a geometric embedding between the ($\mathcal{F}$-type) decomposition spaces $\mathcal{D}(\mathcal{Q},B_1,Y_1)$ and $\mathcal{D}(\mathcal{P},B_2,Y_2)$? Although the answer in general is no, we present criteria for when this holds and examine an illustrative example. \par
Firstly, we need to examine how a quasi-isometric embedding affects the global components of decomposition spaces. Let $Y$ be a sequence space on the countable index set $I$ satisfying the standard assumptions given in Subsection \ref{sec:Definitions_and_Basic_Properties}. Consider two admissible coverings $\mathcal{Q} = (Q_i)_{i \in I}, \mathcal{P} = (P_j)_{j \in J}$ on the sets $X$ and $Z$, respectively. Assume that $\phi:(Z,d_{\mathcal{P}}) \to (X,d_\mathcal{Q})$ is a surjective quasi-isometric embedding. For each $j \in J$ we pick an $i \in I$, denoted by $\phi(j)$, such that $\phi(P_j) \cap Q_i \neq \emptyset$. If this selection can be performed such that each $i \in I$ is picked precisely once, then we say that $\phi$ \textit{induces a bijection between index sets}. If this is so, we define the normed sequence space $(Y_{\phi},\|\cdot\|_{Y_{\phi}})$ by \[Y_{\phi} := \left\{(x_j)_{j \in J} \in \mathbb{C}^{J} \, \Big| \, \left(x_{\phi^{-1}(i)}\right)_{i \in I} \in Y\right\} \, \textrm{with norm} \, \left\|(x_j)_{j \in J}\right\|_{Y_{\phi}} := \left\|\left(x_{\phi^{-1}(i)}\right)_{i \in I}\right\|_{Y}.\] \par
Let us see why the sequence space $Y_{\phi}$ does not depend on the precise choice of bijection that $\phi$ induces: Consider two induced bijections $\phi_1,\phi_{0}:J \to I$ and let $i \in I$ be arbitrary. Then for $j := \phi_{1}^{-1}(i)$ and $l = \phi_{0}^{-1}(i)$ we have $\phi_{1}(P_j) \cap Q_i \neq \emptyset$ and $\phi_{0}(P_l) \cap Q_i \neq \emptyset$. For $x \in \phi_{1}(P_j) \cap Q_i$ and $y \in \phi_{0}(P_l) \cap Q_i$ we use that $\phi$ is a quasi-isometric embedding to obtain \[d_{\mathcal{P}}(z_{x}, z_y) \leq L + C, \quad z_{x} \in \phi^{-1}(x), \, z_{y} \in \phi^{-1}(y).\] Hence there exists a $k = k(L,C) \in \mathbb{N}$ such that $j \in l^{k*}$. The fact that the clustering map $\Gamma_{\mathcal{Q}}$ is bounded on $Y$ ensures the required independence. 
It is straightforward to check that all properties required of the global component of a decomposition space are satisfied for $Y_{\phi}$ if they are satisfied for $Y$.

\begin{theorem}
\label{quasi_implies_embeddings}
Let $\phi:(Z,d_{\mathcal{P}}) \to (X,d_{\mathcal{Q}})$ be a surjective quasi-isometric embedding between the associated metric space of two covered spaces $(X,\mathcal{Q})$ and $(Z,\mathcal{P})$ that induces a bijection between index sets. Consider two ($\mathcal{F}$-type) decomposition spaces $\mathcal{D}(\mathcal{Q},B_1,Y)$ and $\mathcal{D}(\mathcal{P},B_2,Y_{\phi})$ where the local components $B_1$ and $B_2$ consist of functions on $X$ and $Z$, respectively. 
Assume that the mapping \[\phi^{*}f(y) = f(\phi(y))\] between $B_1$ and $B_2$ is bounded. Then $\phi$ induces a geometric embedding from $\mathcal{D}(\mathcal{Q},B_1,Y)$ to $\mathcal{D}(\mathcal{P},B_2,Y_{\phi})$ on the form \[\phi^{*}f := \sum_{i \in I}\phi^{*}(f \cdot \varphi_{i}),\] where $\Phi = (\varphi_{i})_{i \in I}$ is any choice of $\mathcal{Q}$-BAPU.
\end{theorem}

\begin{proof}
Let us fix a $\mathcal{Q}$-BAPU $\Phi = (\varphi_{i})_{i \in I}$ and write $f = \sum_{i \in I}f \cdot \varphi_{i}$ for each $f \in \mathcal{D}(\mathcal{Q},B_1,Y)$ by Proposition \ref{basic_properties}. Then using that $\phi$ induces a bijection between the index sets allows us to write \[\phi^{*}f = \sum_{i \in I}\phi^{*}(f \cdot \varphi_i) = \sum_{j \in J}\phi^{*}(f \cdot \varphi_{\phi(j)}),\] where $\phi^{*}(f \cdot \varphi_{\phi(j)}) \in B_{2}$ by the boundedness of $\phi^{*}$. We want to apply the last statement Proposition \ref{basic_properties} to conclude that $\phi^{*}f \in \mathcal{D}(\mathcal{P},B_2,Y_{\phi})$. To do this, we need to first check that the support condition is satisfied. \par 
We denote as usual the quasi-isometric parameters of $\phi$ by $L,C > 0$. Let $j \in J$ be arbitrary and set $i := \phi(j)$. Since $\phi(P_j) \cap Q_i \neq \emptyset$ we can find $y_i \in P_j \subset Z$ such that $\phi(y_i) \in Q_i$.
Then the constraint $d_{\mathcal{P}}(y,y_i) > L(C+1)$ on $y \in Z$ ensures that $\phi(y) \not\in Q_i$ since we have \[d_{\mathcal{Q}}(\phi(y),\phi(y_i)) \geq \frac{1}{L}d_{\mathcal{P}}(y,y_i) - C > 1.\] Hence \[\textrm{supp}\left(\phi^{*}(f \cdot \varphi_{\phi(j)})\right) =  \textrm{supp}\left(\phi^{*}(f \cdot \varphi_{i})\right) \subset \left\{y \in Z \, \Big| \, d_{\mathcal{P}}(y,y_i) \leq L(C+1) \right\} \subset P_{j}^{k*},\] for some fixed $k = k(C,L) \in \mathbb{N}$. The equivalence \[\left(\|\phi^{*}(f \cdot \varphi_{\phi(j)})\|_{B_2}\right)_{j \in J} \in Y_{\phi} \Leftrightarrow \left(\|\phi^{*}(f \cdot \varphi_{i})\|_{B_2}\right)_{i \in I} \in Y\] together with the boundedness of $\phi^{*}:B_1 \to B_2$ ensure that we can apply the last statement of Proposition \ref{basic_properties} to obtain $\phi^{*}f \in \mathcal{D}(\mathcal{P},B_2,Y_{\phi})$. Moreover, the boundedness of $\phi$ implies that there exists a constant $S > 0$ such that \[\|\phi^{*}f\|_{\mathcal{D}(\mathcal{P},B_2,Y_{\phi})} \leq S\|f\|_{\mathcal{D}(\mathcal{Q},B_1,Y)}.\] \par 
To show injectivity of $\phi^{*}$ we make the following observation: For $f \in \mathcal{D}(\mathcal{Q},B_1,Y)$ we have $f \cdot \varphi_{i} \in B_1$ as a genuine function. Since $\sum_{i \in I}f \cdot \varphi_i = f$ in the norm of $\mathcal{D}(\mathcal{Q},B_1,Y)$ by Proposition \ref{basic_properties}, we can make sense of $f$ as a function on $X$. Assume that $\phi^{*}f = 0$. Then \[0 = \sum_{i \in I}\phi^{*}(f \cdot \varphi_i) = \sum_{i \in I}(f \circ \phi) \cdot (\varphi_{i} \circ \phi).\] Since $(\varphi_{i} \circ \phi)_{i \in I}$ is a partition of unity on $Z$ we have that $f \circ \phi$ is the zero function on $Z$. Thus the surjectivity of  $\phi$ implies that $f$ is the zero function on $X$. Hence $f = 0$ in $B_1$ and injectivity follows. 
\end{proof}

\begin{remark}
There are several ways of modifying the statement in Theorem \ref{quasi_implies_embeddings} to obtain useful variants. To illustrate this, let us consider $B_1 = L^{p}$ and $B_2 = L^{q}$ for $1 \leq p,q < \infty$ on the spaces $X = \mathbb{R}^n$ and $Z = \mathbb{R}^m$. Since the spaces $B_1$ and $B_2$ consist of equivalence classes of functions and not functions themselves, we can not apply Theorem \ref{quasi_implies_embeddings} in this setting. A closer look at the proof of injectivity of $\phi^{*}$ above shows that we the only thing we can conclude from the statement $\phi^{*}f = 0$ in $B_2 = L^{q}$ is that $\phi^{*}f$ is zero almost everywhere as a function on $\mathbb{R}^{m}$. If we add the assumption that $\phi: Z = \mathbb{R}^m \to X = \mathbb{R}^n$ should map sets with measure zero to sets with measure zero (with respect to the respective Lebesgue measures), then the following argument carries through: If $\phi^{*}f = 0$ in $B_2$ then $f \circ \phi$ is zero on a set $Z \setminus N \subset Z$ where $N$ has measure zero.  Then $X = \phi\left(Z  \setminus N\right) \cup \phi(N)$ due to the surjectivity of $\phi$ and $\phi(N)$ has measure zero. Hence $f$ is zero almost everywhere and hence represents the equivalence class of the zero function in $L^{p}$. Therefore $\phi^{*}$ is injective. The assumption that $\phi$ should preserve sets with Lebesgue measure zero is easily satisfied in concrete situations.  
\end{remark}

We will refer to the geometric embeddings arising from the procedure in Theorem \ref{quasi_implies_embeddings} as being \textit{spatially implemented}. It should be remarked that not all geometric embeddings need to be spatially implemented; an example will be given in Theorem \ref{modulation_space_result}. Since surjective quasi-isometric embeddings are quasi-isometries, we can only hope to find spatially implemented geometric embeddings between decomposition spaces $\mathcal{D}(\mathcal{Q},B_1,Y_1)$ and $\mathcal{D}(\mathcal{P},B_2,Y_2)$ whenever $(X,d_{\mathcal{Q}}) \simeq (Z,d_{\mathcal{P}})$. Looking back at Example \ref{dyadic covering} gives an obvious candidate that we now examine. \par
Consider the decomposition space 
\begin{equation}
\label{space_side_Besov}
    \mathbb{B}_{p,q}^{s}(\mathbb{R}^n) := \mathcal{D}\left(\mathcal{B},L^{p},l_{\omega(s)}^{q}\right),
\end{equation} for $1 \leq p,q < \infty$ where $\mathcal{B}(\mathbb{R}^n)$ is the dyadic covering on $\mathbb{R}^n$ given in Example \ref{dyadic covering} and $\omega(s)$ is the weight $\omega(s)(j) = 2^{js}$ for $j \in \mathbb{N}_{0}$. We denote by $\mathbb{B}_{p,q}^{s}(\mathbb{R}_{+})$ the decomposition space whose underlying covered space $(\mathbb{R}_{+},\mathcal{B}(\mathbb{R}_{+}))$ is the positive line with the restricted dyadic covering and the local and global components are the same as in (\ref{space_side_Besov}). The notation $\mathbb{B}_{p,q}^{s}(\mathbb{R}^n)$ is motivated by the fact that the \textit{(inhomogeneous) Besov spaces} $B_{p,q}^{s}(\mathbb{R}^n)$ appearing in classical harmonic analysis have the $\mathcal{F}$-type decomposition space description \[B_{p,q}^{s}(\mathbb{R}^n) = \mathcal{D}^{\mathcal{F}}\left(\mathcal{B},L^{p},l_{\omega(s)}^{q}\right).\] The reason we consider $\mathbb{B}_{p,q}^{s}(\mathbb{R}^n)$ instead of the Besov spaces $B_{p,q}^{s}(\mathbb{R}^n)$ is because we can then use Theorem \ref{quasi_implies_embeddings} to obtain a spatially implemented geometric embedding.

\begin{proposition}
There is a spatially implemented geometric embedding from $\mathbb{B}_{p,q}^{s}(\mathbb{R}_{+})$ to $\mathbb{B}_{p,q}^{ns}(\mathbb{R}^n)$ for any $n \geq 1$.
\end{proposition}
\begin{proof}
To invoke Theorem \ref{quasi_implies_embeddings} we define a map \begin{align*}
    \phi:\mathbb{R}^n & \longrightarrow \mathbb{R}_{+} \\
    x = (x_1, \dots, x_n) & \longmapsto (x_{1}^{2} + \dots + x_{n}^{2})^{\frac{n}{2}} = \|x\|_{2}^{n}.
\end{align*}
The first step is to show that $\phi$ is a quasi-isometry. Associate to any $x \in \mathbb{R}^n$ the smallest number $m(x) \in \mathbb{N}_{0}$ such that $\|x\|_2 \leq 2^{m(x)}$. It is clear that the distance $d_{\mathcal{B}(\mathbb{R}^n)}(x,y)$ between two points $x,y \in \mathbb{R}^n$ satisfies \[d_{\mathcal{B}(\mathbb{R}^n)}(x,y) = d_{\mathcal{B}(\mathbb{R}^n)}((2^{m(x)}, \dots, 0),(2^{m(y)},\dots,0)) + \alpha = |m(x)-m(y)| + \alpha,\] where $\alpha$ will denote a constant that is either one or zero (consider when $x$ and $y$ are in the same dyadic interval to see the necessity of $\alpha$). Then we have
\begin{align*}
    d_{\mathcal{B}(\mathbb{R}_{+})}(\phi(x),\phi(y)) = d_{\mathcal{B}(\mathbb{R}_{+})}\left(2^{m(x)n},2^{m(y)n}\right) + \alpha = n|m(x) - m(y)| + \alpha.
\end{align*}
This is clearly a quasi-isometric embedding with parameters $L = n$ and $C = 1$. It is also clear that $\phi(\mathbb{R}^n)$ is all of $\mathbb{R}_{+}$ by considering the image of any line through the origin. Hence $\phi$ is a surjective quasi-isometry. \par 
However, $\phi$ induces the map $\mathbb{N}_0 \ni m \mapsto nm \in \mathbb{N}_0$ between the index sets. Since this is not a bijection (unless $n = 1$) we need to make the following modification: Scale the dyadic covering on $\mathbb{R}^n$ so that the dyadic intervals have the form \[\widetilde{D_0} := \left\{x \in \mathbb{R}^n \, \Big | \, \|x\|_{2} \leq 2^{\frac{1}{n}} \right\}, \qquad \widetilde{D_{m}} := \left\{x \in \mathbb{R}^n \, \Big | \, 2^{\frac{m-1}{n}} \leq \|x\|_{2} \leq 2^{\frac{m+1}{n}}\right\}.\] The scaled dyadic covering still defines the same decomposition space $B_{p,q}^{s}(\mathbb{R}^n)$ and the map $\phi$ satisfies $\phi(\widetilde{D_{l}}) = D_l$ for all $l \in \mathbb{N}_0$. Hence we obtain that $\phi$ induces a bijection between index sets and the correct sequence space on $\mathbb{R}^n$ is $\left(l_{\omega(s)}^{q}\right)_{\phi} = l_{\omega(ns)}^{q}$. \par 
We can apply Theorem \ref{quasi_implies_embeddings} as longs as we can show that the mapping $\phi^{*}f(y) = f(\phi(y))$ between $L^{p}(\mathbb{R}_+)$ and $L^{p}(\mathbb{R}^n)$ is both bounded above and below. A computation using spherical coordinates gives that  
\begin{align*}
    \|\phi^{*}(f)\|_{L^{p}(\mathbb{R}^n)} & = \left(\int_{\mathbb{R}^n} \left|f\left(\left(x_1^2 + \cdots + x_n^2\right)^{\frac{n}{2}}\right)\right|^{p} \, dx_1 \cdots dx_n\right)^{\frac{1}{p}} \\ & = \left(\int_{0}^{2\pi} \int_{0}^{\pi} \cdots \int_{0}^{\pi} \int_{0}^{\infty}|f(r^{n})|^{p}r^{n-1}\sin^{n-2}(\psi_{1})\cdots \sin(\psi_{n-2}) \, drd\psi_{1} \cdots d \psi_{n-1}\right)^{\frac{1}{p}} \\ & = \left(\frac{n\pi^{\frac{n}{2}}}{\Gamma(1 + \frac{n}{2})}\right)^{\frac{1}{p}}\left(\int_{0}^{\infty}|f(r^n)|^{p}r^{n-1}dr\right)^{\frac{1}{p}} \\ & = \left(\frac{\pi^{\frac{n}{2}}}{\Gamma(1 + \frac{n}{2})}\right)^{\frac{1}{p}} \|f\|_{L^{p}(\mathbb{R}_+)},
\end{align*}
where $\Gamma$ denotes the Gamma function. Hence Theorem \ref{quasi_implies_embeddings} implies that $\phi^{*}$ is a spatially implemented geometric embedding from $\mathbb{B}_{p,q}^{s}(\mathbb{R}_+)$ to $\mathbb{B}_{p,q}^{ns}(\mathbb{R}^n)$. 
\end{proof}

\section{Examples}
\label{sec: Concrete Settings}

In this final section we will put our developed machinery to the test in concrete settings. We will consider the modulation spaces, both on $\mathbb{R}^n$ and on the Heisenberg group $\mathbb{H}_{2n+1}$; the latter case was recently considered in \cite{David}. Finally, we describe a class of decomposition spaces in Subsection \ref{sec:A Decomposition Space of Hyperbolic Type} where the underlying covering is quasi-hyperbolic.

\subsection{Euclidean Modulation Spaces}
\label{sec: Euclidean_Modulation_Spaces}

Modulation spaces are a class of function spaces in time-frequency analysis that have been extensively studied in the last decades. They were introduced by Hans Feichtinger and is widely recognized as the correct setting for theoretical time-frequency analysis after its appearance in the standard reference on the topic \cite{TF_analysis}. The original description was given by Feichtinger in the language of decomposition spaces, while the modern approach is usually through integrability of the short-time Fourier transform. We will begin by giving a brief review of the modern approach and relate it to the decomposition space picture. In Theorem \ref{modulation_space_result} we show that geometric embeddings between modulation spaces in different dimensions can only exist when the dimension is increasing.   \par
The two fundamental operators in time-frequency analysis are the \textit{time-shift} operator $T_x$ and the \textit{frequency-shift} operator $M_{\omega}$. They act on $f \in L^{2}(\mathbb{R}^n)$ by \[T_{x}f(t) = f(t - x), \quad M_{\omega}f(t) = e^{2\pi i t \cdot \omega}f(t), \quad x,\omega \in \mathbb{R}^n.\] Given two functions $f,g \in L^{2}(\mathbb{R}^n)$ where $g \neq 0$ we define the \textit{short-time Fourier transform} (STFT) of $f$ with respect to $g$ to be \begin{equation}
\label{STFT}
    V_{g}f(x,\omega) = \int_{\mathbb{R}^n}f(t)\overline{g(t-x)}e^{-2\pi i t \cdot \omega} \, dt = \langle f, M_{\omega}T_{x}g \rangle_{L^{2}(\mathbb{R}^n)}.
\end{equation} This gives us localized frequency information about $f$ by looking through the \textquote{window} $g$. It is clear from the inner-product interpretation in (\ref{STFT}) that we can extend the STFT to the setting where $f \in \mathcal{S}'(\mathbb{R}^n)$ and $g \in \mathcal{S}(\mathbb{R}^n)$ by duality.

\begin{definition}
\label{Euclidean_modulation_spaces}
Fix $g \in \mathcal{S}(\mathbb{R}^n)\setminus \{0\}$ and constants $1 \leq p,q < \infty$. The \textit{(non-weighted) modulation space} $M^{p,q}(\mathbb{R}^n)$ consists of all tempered distributions $f \in \mathcal{S}'(\mathbb{R}^n)$ that satisfies $\|f\|_{M^{p,q}(\mathbb{R}^n)} < \infty$, where 
\begin{equation}
\label{modulation_norm}
    \|f\|_{M^{p,q}(\mathbb{R}^n)} := \|V_{g}f\|_{L^{p,q}} = \left(\int_{\mathbb{R}^n} \left(\int_{\mathbb{R}^n}|V_{g}f(x,\omega)|^{p} \, dx \right)^{\frac{q}{p}} \, d\omega \right)^{\frac{1}{q}}.
\end{equation}
\end{definition}
It follows from \cite[Proposition 11.3.2]{TF_analysis} that different choices of functions $g \in \mathcal{S}(\mathbb{R}^n) \setminus \{0\}$ yield equivalent norms. Moreover, the spaces $M^{p,q}(\mathbb{R}^n)$ are Banach spaces where the time-shift operators and the frequency-shift operators act by isometries \cite[Theorem 11.3.5]{TF_analysis}. \par 
The modulation spaces have, in addition to their STFT-description, a presentation as $\mathcal{F}$-type decomposition spaces
\begin{equation}
\label{decomposition_description_modulation_spaces}
M^{p,q}(\mathbb{R}^n) \simeq \mathcal{D}^{\mathcal{F}}(\mathbb{R}^n,L^{p},l^{q}).
\end{equation}
One refers to the description of $M^{p,q}(\mathbb{R}^n)$ given in Definition \ref{Euclidean_modulation_spaces} as the \textit{coorbit description} of $M^{p,q}(\mathbb{R}^n)$, while (\ref{decomposition_description_modulation_spaces}) is referred to as the \textit{decomposition description} of $M^{p,q}(\mathbb{R}^n)$. 

\begin{theorem}
\label{modulation_space_result}
There is a tower of compatible geometric embeddings \[M^{p,q}(\mathbb{R}) \xrightarrow{\Gamma_{1}^{2}}M^{p,q}(\mathbb{R}^2) \xrightarrow{\Gamma_{2}^3} \dots \xrightarrow{\Gamma_{n-1}^{n}} M^{p,q}(\mathbb{R}^n) \xrightarrow{\Gamma_{n}^{n+1}} \dots, \] where there are no geometric embeddings in the other direction.
\end{theorem}

\begin{proof}
It follows from Proposition \ref{restrictions_on_embeddings} and Remark \ref{remark_geometric_restrictions} that there are no geometric embeddings from $M^{p,q}(\mathbb{R}^n)$ to $M^{p,q}(\mathbb{R}^m)$ whenever $n > m$. We will now show that $M^{p,q}(\mathbb{R}^{n})$ can be geometrically embedded into $M^{p,q}(\mathbb{R}^{m})$ as long as $n \leq m$. \par 
Define a map \begin{align*}
\Gamma_{n}^{m}: \mathcal{S}\left(\mathbb{R}^n\right) \subset M^{p,q}\left(\mathbb{R}^n\right) & \longrightarrow M^{p,q}\left(\mathbb{R}^m\right) \\
f & \longmapsto \Gamma_{n}^{m}(f)(\xi_1, \dots, \xi_m) := \mathcal{F}_{m}^{-1}\left(\mathcal{F}_{n}(f)(\xi_1,\dots,\xi_n)\eta(\xi_{n+1})\cdots\eta(\xi_m)\right),
\end{align*}
where $0 \neq \eta \in C_{c}^{\infty}(\mathbb{R})$ and $\mathcal{F}_{n}$ denotes the $n$-dimensional Fourier transform. It is clear that the condition (\ref{geometric_condition}) is satisfied. Since $\mathcal{S}(\mathbb{R}^{n})$ is dense in $M^{p,q}(\mathbb{R}^n)$ by \cite[Theorem 12.2.2]{TF_analysis} it suffices to show boundedness of $\Gamma_{n}^{m}$. To show this, we utilize the coorbit description of $M^{p,q}(\mathbb{R}^n)$. Since the Fourier transform interchanges time-shift operators and frequency-shift operators, it follows that $\mathcal{F}_{n}$ maps $M^{p,q}(\mathbb{R}^n)$ boundedly into $M^{q,p}(\mathbb{R}^n)$. Hence it suffices to show that the map $f \mapsto f\otimes \eta$ is a bounded map from $\mathcal{S}(\mathbb{R}^n) \subset M^{p,q}(\mathbb{R}^n)$ to $M^{p,q}(\mathbb{R}^{n+1})$ whenever $0 \neq \eta \in C_{c}^{\infty}(\mathbb{R})$ and $1 \leq p,q < \infty$. \par 
The standard Gaussian $g_{n+1}(x) := e^{-\pi x^2}$ on $\mathbb{R}^{n+1}$ splits as $g_{n+1}(x) = (g_{n}\otimes g_{1})(x) := g_{n}(\bar{x})g_{1}(x_{n+1}),$ where $x = (x_1, \dots, x_{n+1})$ and $\bar{x} = (x_1, \dots, x_n)$. Hence \[V_{g_{n+1}}(f \otimes \eta)(x,\omega) = V_{g_{n} \otimes g_{1}}(f \otimes \eta)(x,\omega) = V_{g_{n}}f(\bar{x},\bar{\omega})\cdot V_{g_{1}}\eta(x_{n+1},\omega_{n+1}),\] and a straightforward calculation gives that \begin{align*}
    \|f \otimes \eta\|_{M^{p,q}(\mathbb{R}^{n+1})} & = \left(\int_{\mathbb{R}^{n+1}}\left(\int_{\mathbb{R}^{n+1}}\left|V_{g_{n+1}}(f \otimes \eta)(x,\omega)\right|^{p} \, dx\right)^{\frac{q}{p}} \, d\omega \right)^{\frac{1}{q}} \\ & = 
    \left(\int_{\mathbb{R}^{n+1}}\left(\int_{\mathbb{R}^{n}}\left|V_{g_{n}}f(\bar{x},\bar{\omega})\right|^{p} \, d\bar{x} \int_{\mathbb{R}} \left|V_{g_{1}}\eta(x_{n+1},\omega_{n+1})\right|^{p}\, dx_{n+1} \right)^{\frac{q}{p}} \, d\omega \right)^{\frac{1}{q}} \\ & = \left(\int_{\mathbb{R}}\left(\int_{\mathbb{R}}\left|V_{g_{1}}\eta(x_{n+1},\omega_{n+1})\right|^{p} \, dx_{n+1}\right)^{\frac{q}{p}} \, d\omega_{n+1} \right)^{\frac{1}{q}} \cdot \|f\|_{M^{p,q}(\mathbb{R}^n)}.
\end{align*}
Since $0 \neq \eta \in C_{c}^{\infty}(\mathbb{R}) \subset \mathcal{S}(\mathbb{R}) \subset M^{p,q}(\mathbb{R})$ it follows that $\Gamma_{n}^{m}$ is a bounded map from $M^{p,q}(\mathbb{R}^n)$ to $M^{p,q}(\mathbb{R}^m)$. \par 
The reason $\Gamma$ is injective when viewed as a mapping from $M^{p,q}(\mathbb{R}^n)$ to $M^{p,q}(\mathbb{R}^m)$ is because the Fourier transform is an injective map from $M^{p,q}(\mathbb{R}^n)$ to $M^{q,p}(\mathbb{R}^n)$ and that $\eta \neq 0$. Hence $\Gamma_{n}^{m}$ extends to a geometric embedding from $M^{p,q}(\mathbb{R}^n)$ to $M^{p,q}(\mathbb{R}^m)$ for $n \leq m$. Finally, the embeddings we constructed respect composition $\Gamma_{m}^{l} \circ \Gamma_{n}^{m} = \Gamma_{n}^{l}$ for all $l \geq m \geq n \geq 1$. 
\end{proof}

We can say even more by allowing the indices $1 \leq p,q < \infty$ to vary. It follows from \cite[Theorem 12.2.2]{TF_analysis} that we have the estimate \[\|f\|_{M^{p_2,q_2}(\mathbb{R}^k)} \leq A \|f\|_{M^{p_1,q_1}(\mathbb{R}^k)}\] for some $A > 0$, whenever $p_1 \leq p_2$ and $q_1 \leq q_2$. 

\begin{corollary}
\label{Feichtinger_Corollary}
There exists a geometric embedding from $M^{p_1,q_1}(\mathbb{R}^n)$ to $M^{p_2,q_2}(\mathbb{R}^m)$ whenever $p_1 \leq p_2$, $q_1 \leq q_2$ and $n \leq m$. In particular, there exists a geometric embedding from the Feichtinger algebra $\mathcal{S}_{0}(\mathbb{R}) := M^{1,1}(\mathbb{R})$ to any modulation space $M^{p,q}(\mathbb{R}^n)$.  
\end{corollary} 

Hence the Feichtinger algebra is universal in the class of (non-weighted) modulation spaces on Euclidean spaces. Therefore, any ($\mathcal{F}$-type) decomposition space that embeds geometrically into $\mathcal{S}_{0}(\mathbb{R})$ does in fact embed geometrically into all the modulation spaces $M^{p,q}(\mathbb{R}^n)$.

\subsection{Heisenberg Modulation Spaces}
\label{sec: David}

The STFT introduced in (\ref{STFT}) is intimately related to the Heisenberg group $\mathbb{H}_{2n+1}$ in the following way: Define the \textit{Schr\"odinger representation} $\rho:\mathbb{H}_{2n+1} \to \mathcal{U}(L^{2}(\mathbb{R}^n))$ by \[\rho(x,\omega,t) = e^{\pi i\left(2t + x \cdot \omega\right)}T_{x}M_{\omega},\] where $x, \omega \in \mathbb{R}^n$, $t \in \mathbb{R}$ and $\mathcal{U}(L^{2}(\mathbb{R}^n))$ denotes the unitary operators on $L^{2}(\mathbb{R}^n)$. Then a short computations shows that the matrix coefficients of the Schr\"odinger representation are (up to a phase factor) the STFT. The Stone-von Neumann theorem \cite[Theorem 9.3.1]{TF_analysis} emphasizes the importance of the Schr\"odinger representation as it is essentially the only interesting irreducible unitary representation of the Heisenberg group. \par 
Although it is clear from Definition \ref{Euclidean_modulation_spaces} that the Heisenberg group $\mathbb{H}_{2n+1}$ play a role in the traditional modulation spaces, the underlying covering of $M^{p,q}(\mathbb{R}^n)$ has $\mathbb{Z}^n$ as its associated metric space and not the discrete Heisenberg groups. Recently, decomposition spaces originating from a coorbit description of a certain nilpotent Lie group have been investigated in \cite{David}. These decomposition spaces are truly related to the large scale geometry of the Heisenberg group. We outline their construction and extend one of their main results \cite[Theorem 7.6]{David} to geometric embeddings in Proposition \ref{extension_of_David_result} since all the hard work has already been done in Section \ref{Chapter_Uniform_Metric_Spaces} and Section \ref{Chapter_Decomposition_Spaces}. We believe that our approach can make arguments clearer and emphasize the importance of viewing coverings from a metric perspective. Thus we are able to approach some of the novel results in \cite{David} from a different angle because of our large scale machinery. 
\par
The \textit{(abstract) Dynin-Folland Lie algebra} $\mathfrak{h}_{n,2}$ is the nilpotent Lie algebra with basis 
\begin{equation}
\label{basis_Dynin_Folland}
    \langle X_{u_1}, \dots, X_{u_n},X_{v_1}, \dots, X_{v_n}, X_{w}, X_{x_1}, \dots, X_{x_n},X_{y_1}, \dots, X_{y_n},X_z,X_s \rangle,
\end{equation} and with non-vanishing commutation relations 
\begin{align*}
    [X_{u_j},X_{v_k}]_{\mathfrak{h}_{n,2}} & = \delta_{j,k} X_w, \quad  [X_{u_j},X_{x_k}]_{\mathfrak{h}_{n,2}} = \delta_{j,k} X_s, \quad [X_{u_j},X_{z}]_{\mathfrak{h}_{n,2}} = -\frac{1}{2}X_{y_j}, \\ [X_{v_j},X_{y_k}]_{\mathfrak{h}_{n,2}} & = \delta_{j,k} X_s, \quad \, [X_{v_j},X_{z}]_{\mathfrak{h}_{n,2}} = \frac{1}{2} X_{x_j}, \quad \, \, \, \, \, [X_w,X_z]_{\mathfrak{h}_{n,2}} = X_s,
\end{align*}
where $j,k = 1, \dots, n$. The first $2n+1$ basis vectors in (\ref{basis_Dynin_Folland}) generate a subalgebra that is isomorphic to the Lie algebra of the Heisenberg group $\mathbb{H}_{2n+1}$. We denote by $\mathbf{H}_{n,2}$ the connected and simply connected Lie group corresponding to $\mathfrak{h}_{n,2}$ called the \textit{Dynin-Folland group}. \par 
In \cite[Theorem 4.5 and Corollary 4.7]{David} the authors classify all the irreducible and projective representations of the Dynin-Folland group by using Kirillov's orbit method. One of these projective representations is used to define the \textit{Heisenberg modulation spaces} similarly to how the Schr\"odinger representation is used to define the modulation spaces $M^{p,q}(\mathbb{R}^n)$. We refer the reader to \cite{David} for the explicit description as we will only need the decomposition space description of the Heisenberg modulation spaces. \par
In \cite{David} they consider the lattice in $\mathbb{H}_{2n+1} \simeq \mathbb{R}^{2n+1}$ defined by \[\Gamma := \left\{(a,b,c) \in \mathbb{R}^{2n+1} \, \Big | \, a,b \in (2\mathbb{Z})^n, \, c \in 2\mathbb{Z} \right\}.\] From this a covering $\mathcal{P}$ on $\mathbb{H}_{2n+1} \simeq \mathbb{R}^{2n+1}$ is induced by defining \[\mathcal{P} := \left\{P * \gamma \, \Big| \, P = (-\epsilon,2 + \epsilon)^{2n+1}, \, \gamma \in \Gamma\right\},\] where $\epsilon \in (0,\frac{1}{2})$ and the multiplication $P * \gamma$ is with the Heisenberg group structure. Define the $\mathcal{F}$-type decomposition spaces \[E^{p,q}\left(\mathbb{H}_{2n+1}\right) := \mathcal{D}^{\mathcal{F}}(\mathcal{P},L^p,l^q),\] where $1 \leq p,q < \infty$ and the reservoir is the tempered distributions. \par 
We remark that \cite{David} consider the spaces with weights derived from the \textit{homogeneous Cygan-Koranyi norm}  \[(p,q,t) \longmapsto \left((|p|^2 + |q|^2)^2 + 16t^2\right)^{\frac{1}{4}}.\] We omit this extension as all the geometric features are already present in the case without weights. Moreover, we refer the reader to \cite[Theorem 7.3]{David} where the authors show that the spaces $E^{p,q}\left(\mathbb{H}_{2n+1}\right)$ coincide with the Heisenberg modulation spaces arising from the projective representations of the Dynin-Folland group $\mathbf{H}_{n,2}$. 

\begin{proposition}
\label{extension_of_David_result}
None of the spaces $E^{p,q}\left(\mathbb{H}_{2n+1}\right), M^{p,q}(\mathbb{R}^k)$, and $B_{p,q}^{s}(\mathbb{R}^l)$ are geometrically isomorphic for any values $n,k,l \geq 1$, $p,q \in [1,\infty)$. 
\end{proposition}

\begin{proof}
It is clear from the results in Section \ref{Chapter_Uniform_Metric_Spaces} that the covering $\mathcal{P}$ is the uniform covering on $\mathbb{H}_{2n+1}$. Since any lattice in a stratified Lie group is uniform, the lattice $\Gamma$ is a net in $(\mathbb{H}_{2n+1},d_{\mathcal{P}})$. Thus Proposition \ref{Carnot_group_result} implies that the uniform metric space $(\mathbb{H}_{2n+1},d_{\mathcal{P}})$ is quasi-isometric $\Gamma$ equipped with any proper, left-invariant metric. \par 
The fact that $E^{p,q}\left(\mathbb{H}_{2n+1}\right)$ and $M^{p,q}(\mathbb{R}^k)$ are not geometrically isomorphic follows from Proposition \ref{restrictions_on_embeddings}. The order of the polynomial growth of $\Gamma$ is $2n+2$ by Theorem \ref{growth_vector_result}, while the growth of the underlying covering of the Besov space $B^{p,q}(\mathbb{R}^l)$ is linear. Hence the spaces $E^{p,q}\left(\mathbb{H}_{2n+1}\right)$ and $B_{p,q}^{s}(\mathbb{R}^l)$ are not geometrically isomorphic by Proposition \ref{covering_invariance} since growth type is a quasi-isometric invariant. Finally, the modulation spaces $M^{p,q}(\mathbb{R}^k)$ and Besov spaces $B_{p,q}^{s}(\mathbb{R}^l)$ are not geometrically isomorphic by Proposition \ref{covering_invariance} and Example \ref{example_dyadic_and_grid}.
\end{proof}

Notice that for $p = q = 2$, all three spaces $M^{2,2}(\mathbb{R}^{2n+1})$, $B_{2,2}(\mathbb{R}^{2n+1})$ and $E^{2,2}\left(\mathbb{H}_{2n+1}\right)$ are all simply $L^{2}(\mathbb{R}^{2n+1})$ as Banach spaces by \cite[Lemma 6.10]{Felix_main}. However, the identity map 
\[Id:M^{2,2}(\mathbb{R}^{2n+1}) \longrightarrow E^{2,2}\left(\mathbb{H}_{2n+1}\right)\] 
is not a geometric isomorphism between $\mathcal{F}$-type decomposition spaces since the associated metric spaces of the underlying coverings are not quasi-isometric. Hence geometric isomorphisms incorporate the coverings and thus treat decomposition spaces as more than Banach spaces.

\subsection{A Decomposition Space of Hyperbolic Type}
\label{sec:A Decomposition Space of Hyperbolic Type}

So far, we have looked at several examples of decomposition spaces that have already been present in the literature. We conclude by examining a new decomposition space having an underlying covering whose associated metric space is quasi-hyperbolic (and not infinite cyclic). \par

\begin{definition}
We call the space \[\mathcal{D}^{p,q}\left(SL(2,\mathbb{R})\right) := \mathcal{D}(SL(2,\mathbb{R}),L^{p},l^q)\] the \textit{hyperbolic decomposition space} with parameters $1 \leq p,q < \infty$. Here $L^{p}$ denotes the (equivalence classes of) $p$'th integrable functions on $SL(2,\mathbb{R})$ with respect to the Haar measure on $SL(2,\mathbb{R})$. Whenever $p = q = 1$, we call $\mathcal{D}\left(SL(2,\mathbb{R})\right) := \mathcal{D}^{1,1}\left(SL(2,\mathbb{R})\right) $ the \textit{standard hyperbolic decomposition space}.
\end{definition}

A few remarks are in order: Since the group $SL(2,\mathbb{R})$ is unimodular we do not need to distinguish between the left and right Haar measure on $SL(2,\mathbb{R})$. Secondly, the terminology \textit{hyperbolic decomposition space} is motivated by Theorem \ref{SL}. Finally, we can take the reservoir to be $\mathcal{A} = C_{b}(SL(2,\mathbb{R}),\mathbb{C})$ as this is of minor importance by \cite[Theorem 1 (ii)]{Wiener}. It follows from Proposition \ref{basic_properties} that $\mathcal{D}^{p,q}\left(SL(2,\mathbb{R})\right)$ is reflexive as a Banach space whenever $1 < p,q < \infty$. 

\begin{example}
\label{Iwasawa example}
Let us, for the sake of concreteness, give an example of an element in the standard hyperbolic decomposition space $\mathcal{D}\left(SL(2,\mathbb{R})\right)$. Every element $\alpha \in SL(2,\mathbb{R})$ has an \textit{Iwasawa decomposition} \[\alpha = \begin{pmatrix} \cos(\theta) & -\sin(\theta) \\ \sin(\theta) & \cos(\theta) \end{pmatrix} \begin{pmatrix} y & x \\ 0 & \frac{1}{y} \end{pmatrix},\] for $0 \leq \theta < 2\pi$, $x \in \mathbb{R}$, and $y > 0$ \cite[Chapter 26]{Bump}. We will write elements in $SL(2,\mathbb{R})$ as $(\theta,x,y)$ according to their Iwasawa decomposition. In these coordinates, the Haar measure on $SL(2,\mathbb{R})$ is given by $y^{-2}dxdyd\theta$. \par Consider the function $f:SL(2,\mathbb{R}) \to \mathbb{R}_+$ given by \[f(\theta,x,y) = y^3e^{-y - x^2}.\] Then a short computation shows that \[\|f\|_{L^{1}} = \int_{SL(2,\mathbb{R})}f(z)d\mu(z) = \int_{0}^{2\pi}\int_{0}^{\infty}\int_{-\infty}^{\infty}y^3e^{-y - x^2}\frac{dx \, dy \, d\theta}{y^2} = 2\pi^{\frac{3}{2}},\] by utilizing the value of the Gamma function at zero. Since $f$ is positive, we have the trivial estimate \[\|f\|_{\mathcal{D}\left(SL(2,\mathbb{R})\right)} \leq N_{\mathcal{U}}2\pi^{\frac{3}{2}},\] where $N_{\mathcal{U}}$ is the admissibility constant of the uniform covering $\mathcal{U}$. Hence $f \in \mathcal{D}\left(SL(2,\mathbb{R})\right)$. 
\end{example}

We will now show that the hyperbolic decomposition space $\mathcal{D}^{p,q}\left(SL(2,\mathbb{R})\right)$ is fundamentally different from the decomposition spaces we previously examined.

\begin{proposition}
\label{final_result}
There are no geometric embeddings 
\begin{align*}
    \phi_n:M^{p,q}(\mathbb{R}^n) & \longrightarrow \mathcal{D}^{p,q}\left(SL(2,\mathbb{R})\right), \qquad && \eta_k:\mathcal{D}^{p,q}\left(SL(2,\mathbb{R})\right) \longrightarrow M^{p,q}(\mathbb{R}^k), \\ \psi_m: E^{p,q}\left(\mathbb{H}_{2m+1}\right) & \longrightarrow \mathcal{D}^{p,q}\left(SL(2,\mathbb{R})\right), \qquad && \tau_l: \mathcal{D}^{p,q}\left(SL(2,\mathbb{R})\right) \longrightarrow E^{p,q}\left(\mathbb{H}_{2l+1}\right), \\
    \theta_d:B_{p,q}^{s}(\mathbb{R}^d) & \longrightarrow \mathcal{D}^{p,q}\left(SL(2,\mathbb{R})\right), \qquad && \sigma_r: \mathcal{D}^{p,q}\left(SL(2,\mathbb{R})\right)  \longrightarrow B_{p,q}^{s}(\mathbb{R}^r),
\end{align*}
for $n \geq 2$ and $m,d,k,l,r \geq 1$. However, for $n = 1$ the Feichtinger algebra $\mathcal{S}_{0}(\mathbb{R}) := M^{1,1}(\mathbb{R})$ embeds geometrically into the standard hyperbolic decomposition space $\mathcal{D}\left(SL(2,\mathbb{R})\right)$. 
\end{proposition}

\begin{proof}
The fact that $\phi_n$, $\psi_m$, and $\theta_d$ can not be geometric embeddings (unless $n = 1$) follows from the hyperbolicity of $(SL(2,\mathbb{R}),d_{\mathcal{U}})$ together with Proposition \ref{covering_invariance} and Lemma \ref{hyperbolicity_lemma} (a). If we assume that $\eta_k$ is a geometric embedding, then Proposition \ref{covering_invariance} and Theorem \ref{SL} imply that there is a quasi-isometric embedding between the hyperbolic plane $\mathbb{H}^2$ and $\mathbb{R}^k$. Since $\mathbb{R}^k$ is quasi-isometric to $\mathbb{Z}^k$ and $\mathbb{H}^2$ is quasi-isometric to $\pi_{1}(X)$ by Proposition \ref{SL}, where $X$ is a compact Riemann surface of genus $g \geq 2$, we then have a quasi-isometric embedding \[\widetilde{\eta}_k:\pi_{1}(X) \longrightarrow \mathbb{Z}^k.\] \par 
Any hyperbolic group that is not finite or contain $\mathbb{Z}$ as a finite index subgroup does contains the free group on two generators as a subgroup \cite{Gromov_Hyperbolic}. The free group is easily seen to have exponential growth type. Hence it follows that $\pi_{1}(X)$ also has exponential growth type since any finitely generated group can have at most exponential growth type. On the other hand, the growth type of $\mathbb{Z}^k$ is, as we have mentioned previously, polynomial. Hence the impossibility of $\widetilde{\eta}_k$ follows from \cite[Proposition 6.2.4]{Geometric_Group_Theory}. The same argument works for $\tau_l$ and $\sigma_r$ since the growth types of $\mathbb{H}_{2l+1}(\mathbb{Z})$ and $\mathbb{N}_{0}$ are both polynomial. \par
For the case $n = p = q = 1$, we can define a map $\phi_1: \mathcal{S}_{0}(\mathbb{R}) \longrightarrow \mathcal{D}\left(SL(2,\mathbb{R})\right)$ given by \begin{equation}
\label{embedding_into_hyperbolic_decomposition_space}
    \phi_{1}(f)(\alpha) = \phi_{1}(f)\left(\begin{pmatrix} \cos(\theta) & -\sin(\theta) \\ \sin(\theta) & \cos(\theta) \end{pmatrix} \begin{pmatrix} y & x \\ 0 & \frac{1}{y} \end{pmatrix}\right) = f(x)\eta(y),
\end{equation} 
where $0 \neq \eta \in C_{c}^{\infty}(\mathbb{R})$ is supported in $\left[\frac{1}{2},1\right]$ and we have used the Iwasawa decomposition of $\alpha \in SL(2,\mathbb{R})$. The pointwise evaluation in \eqref{embedding_into_hyperbolic_decomposition_space} is well-defined since every element in $\mathcal{S}_{0}(\mathbb{R})$ is a continuous function \cite[Proposition 12.1.4]{TF_analysis}. Let $\Phi = (\varphi_i)_{i \in I}$ be a $\mathcal{U}$-BAPU for the uniform covering $\mathcal{U}$ on $SL(2,\mathbb{R})$. Then a straightforward computation similar to Example \ref{Iwasawa example} shows that $\phi_{1}(f) \cdot \varphi_i \in L^{1}\left(SL(2,\mathbb{R})\right)$  for every $i \in I$ and \[\left\{\|\phi_{1}(f) \cdot \varphi_i\|_{L^{1}(SL(2,\mathbb{R}))}\right\}_{i \in I} \in l^{1}, \quad \sum_{i \in I}\phi_{1}(f) \cdot \varphi_i = \phi_{1}(f).\] Hence we can conclude from the last statement in Proposition \ref{basic_properties} that $\phi_{1}(f) \in \mathcal{D}\left(SL(2,\mathbb{R})\right)$ and we have \[\|\phi_{1}(f)\|_{\mathcal{D}\left(SL(2,\mathbb{R})\right)} \leq A\|f\|_{\mathcal{S}_{0}(\mathbb{R})},\] where the constant $A > 0$ does not depend on $f \in \mathcal{S}_{0}(\mathbb{R})$. If $f \in \mathcal{S}_{0}(\mathbb{R})$ with $\mathcal{C}[f] \subset [n-k,n+k]$ for $n \in \mathbb{Z}$ and $k \in \mathbb{N}$ then \[\mathcal{C}\left[\phi_{1}(f)\right] \subset (0,2\pi) \times [n-k,n+k] \times \left[\frac{1}{2},1\right],\] with respect to the Iwasawa decomposition. Hence $\phi_{1}$ satisfies (\ref{geometric_condition}) since the map \[\mathbb{Z} \ni n \longmapsto \begin{pmatrix} 1 & n \\ 0 & 1\end{pmatrix} = \begin{pmatrix} 1 & 1 \\ 0 & 1\end{pmatrix}^{n} \in SL(2,\mathbb{R})\] is a quasi-isometric embedding by Proposition \ref{hyperbolicity_lemma} (c). \par Finally, the map $\phi_{1}$ is injective since the bump function $\eta$ is assumed to be non-zero. Thus $\phi_{1}$ is a geometric embedding. It is not a geometric isomorphism since the image of $\phi_1$ does not contain any function that depends on the variable $\theta$ with respect to the Iwasawa decomposition. Moreover, $\phi_{1}$ is not a spatially implemented geometric embedding since $\mathbb{Z}$ is not quasi-isometric to $\mathbb{H}^2$ as they have different asymptotic dimension.
\end{proof}

Since every stratified Lie group $G$ is diffeomorphic to $\mathbb{R}^n$ for some $n \in \mathbb{N}$ we can identify the uniform covering $\mathcal{U}(G)$ with a covering on $\mathbb{R}^n$ where the Fourier transform makes sense. Hence we can consider the decomposition space $\mathcal{D}^{\mathcal{F}}(G,L^p,l^q)$. Both $M^{p,q}(\mathbb{R}^n)$ and $E^{p,q}(\mathbb{H}_{2m+1})$ are particular examples in this class, and one might refer to them as $\mathcal{F}$\textit{-type stratified decomposition spaces}. \par 
In the case where the stratified Lie group is realizable over the rationals, we know from Theorem \ref{growth_vector_result} that the uniform metric space $(G,d_{\mathcal{U}})$ is quasi-isometric to a finitely generated group $N$ with polynomial growth type. Hence the argument used in the first part of the proof of Proposition \ref{final_result} carries through to show that $N$ is not hyperbolic unless $N$ is quasi-isometric to $\mathbb{Z}$. This is only possible for $G = \mathbb{R}$, so $M^{p,q}(\mathbb{R})$ is the only $\mathcal{F}$-type stratified decomposition space built on a quasi-hyperbolic covering. \par 
Thus a straightforward extension of Proposition \ref{final_result} shows that there are no geometric embeddings from $\mathcal{D}^{\mathcal{F}}(G,L^p,l^q)$ to $\mathcal{D}^{p,q}(SL(2,\mathbb{R}))$ or vice versa when $G$ is a stratified Lie group realizable over the rationals that is not $\mathbb{R}$. In particular, this holds for the $\mathcal{F}$-type stratified decomposition space where the stratified Lie group is the Engel group introduced in Example \ref{Engel_group}. Showing statements such as these without using invariants from large scale geometry seems highly non-trivial and highlights the usefulness of our approach.  


\Addresses

\end{document}